\documentclass[oneside, reqno, 11pt]{amsart}
\usepackage{amssymb,mathrsfs,amstext, enumerate, comment}
\usepackage[noadjust]{cite}
\usepackage[plainpages=false,colorlinks,hyperindex,pdfpagemode=None,bookmarksopen,linkcolor=blue,citecolor=blue,urlcolor=blue]{hyperref}

\theoremstyle{proclaim}
\newtheorem{theorem}{Theorem}[section]
\newtheorem{proposition}[theorem]{Proposition}
\newtheorem{corollary}[theorem]{Corollary}
\newtheorem{lemma}[theorem]{Lemma}

\theoremstyle{statement}

\newtheorem{example}[theorem]{Example}
\newcommand{\TT}{{\mathrm{t}}}

\theoremstyle{remark}
\newtheorem{remark}[theorem]{Remark}

\numberwithin{equation}{section}

\def\bH{{\bf H}}
\def\Span{{\rm span}\,}

\def\bR{{\mathbf{R}}}

\def\bC{{\mathbf{C}}}

\def\cC{{\mathcal C}}

\def\bV{{\bf V}}
\def\bU{{\bf U}}
\def\Span{{\rm span}\,}

\def\diag{{\rm diag}\,}
\def\tr{\operatorname{trace}}

\def\dim{{\rm dim}\,}

\def\IC{{\mathbb C}}
\def\IF{{\mathbb F}}
\def\IR{{\mathbb R}}

\def\bM{{\mathbf M}}

\newcommand{\bigzero}{\mbox{\normalfont\Large $0$}}
\newcommand{\rvline}{\hspace*{-\arraycolsep}\vline\hspace*{-\arraycolsep}}

\begin{document}
\openup .535\jot
\title{Linear maps on matrices preserving parallel pairs}

\dedicatory{Dedicated to Professor  Yiu-Tung Poon on the occasion of his 70th birthday}

\author[Li, Tsai, Wang  and Wong]
{Chi-Kwong Li, Ming-Cheng Tsai, Ya-Shu Wang  \and Ngai-Ching Wong}

\address[Li]{Department of Mathematics, The College of William
\& Mary, Williamsburg, VA 13187, USA.}
\email{ckli@math.wm.edu}

\address[Tsai]{General Education Center, National Taipei University of  Technology, Taipei 10608, Taiwan.}
\email{mctsai2@mail.ntut.edu.tw}

\address[Wang]{Department of Applied Mathematics, National Chung Hsing University, Taichung 40227, Taiwan.}
\email{yashu@nchu.edu.tw}

\address[Wong]{Department of Applied Mathematics, National Sun Yat-sen University, Kaohsiung, 80424, Taiwan; Department of Healthcare
Administration and Medical Information,
Kaohsiung  Medical University, 80708 Kaohsiung, Taiwan.}
  \email{wong@math.nsysu.edu.tw}
\date{}

\begin{abstract}
Two (real or complex) $m\times n$ matrices $A$ and $B$ are said to be parallel
(resp.\ triangle equality attaining, or TEA in short)
with respect to the spectral norm $\|\cdot\|$ if
$\|A+ \mu B\| = \|A\| + \|B\|$ for some scalar $\mu$ with $|\mu|=1$ (resp.\ $\mu=1$). We study linear maps
$T$ on $m\times n$  matrices preserving parallel (resp.\ TEA) pairs, i.e.,
$T(A)$ and $T(B)$ are parallel (resp.\ TEA) whenever $A$ and $B$ are parallel (resp.\ TEA).

It is shown that when $m,n \ge 2$ and $(m,n) \ne (2,2)$,
a nonzero linear map $T$ preserving TEA pairs if and only if it is
a positive multiple of a linear isometry, namely, $T$  has the
form

\medskip\centerline{
(1)  \ $A \mapsto \gamma UAV$ \quad or \quad
(2) \ $A \mapsto \gamma UA^{\TT} V$ (in this case, $m = n$),}

\smallskip\noindent
for a positive number $\gamma$, and unitary (or real orthogonal) matrices
$U$ and $V$ of appropriate sizes.  Linear maps preserving
parallel pairs are those carrying form (1), (2), or the form

\medskip\centerline{
(3)  \ $A \mapsto f(A) Z$}

\medskip\noindent
 for a linear functional $f$ and a fixed matrix $Z$.

\medskip
The case when $(m,n) = (2,2)$ is more complicated.
There are  linear maps of $2\times 2$ matrices preserving parallel pairs or TEA pairs neither of the form (1), (2)  nor (3) above.
Complete characterization of such maps is given with some
intricate computation and techniques in matrix groups.
\end{abstract}


\maketitle


\section{Introduction}\label{s:Intro}

Let $(\bV,\|\cdot\|)$ be a normed vector space over either the real field $\IF=\IR$ or
the complex field $\IF=\IC$.  Two vectors $x, y \in \bV$ form a  \emph{parallel} pair if
\begin{equation}\label{parallel}
\|x+\mu y\| = \|x\| + \|y\| \quad \hbox{ for some } \mu \in \IF \hbox{ with } |\mu| = 1;
\end{equation}
the vectors $x,y \in \bV$ form a \emph{triangle equality attaining} (\emph{TEA}) pairs if
\begin{equation}\label{TEA}
\|x+y\| = \|x\| + \|y\|.
\end{equation}
We are interested in the case when $\bV=
 \bM_{m,n}(\IF)$, the linear space of $m\times n$ matrices over $\IF$
equipped with the spectral norm, that is, the operator norm
$$\|A\| = \max \{ \|Ax\|: x \in \IF^n, \|x\| = 1\},$$
while $\IF^m, \IF^n$ are  equipped with the usual inner products and norms.
One may see \cite{NT02,BB01,Seddik07,Wojcik,ZM16,Zamani} for the study of parallel pairs.

In this paper, wee study those linear maps
$T: \bM_{m,n}(\IF) \rightarrow \bM_{m,n}(\IF)$  preserving parallel pairs, that is,
\begin{equation}\label{condition-p1}
\max_{|\nu|=1} \|T(A)+ \nu T(B)\| = \|T(A)\| + \|T(B)\| \quad \hbox{whenever} \quad \max_{|\mu|=1} \|A+\mu B\| = \|A\|+\|B\|,
\end{equation}
or preserving TEA pairs, that is,
\begin{equation} \label{condition-p2}
\|T(A)+T(B)\| = \|T(A)\| + \|T(B)\| \quad \hbox{whenever } \quad \|A+B\| = \|A\|+\|B\|.
\end{equation}
Evidently, if a linear map
 $T$ satisfies (\ref{condition-p2}), then it
also satisfies  (\ref{condition-p1}).
In other words, linear TEA pair preservers are also  parallel pair preservers.

Note that if $m = 1$ or $n = 1$,
the spectral norm  reduce to the $\ell_2$-norm, which is strictly convex.
Consequently, two nonzero   row matrices or
 column matrices
  $A, B$ are parallel (resp.\ form a TEA pair)
  exactly when $A=\lambda B$ for some scalar $\lambda$ (resp.\ with $\lambda >0$).
  Thus any linear map of row matrix spaces or column matrix spaces preserves parallel pairs as well as TEA pairs.
We will assume that $m,n \ge 2$ throughout in this paper.

Note that  $T$ preserves TEA/parallel pairs if $T: \bM_{m,n}(\IF) \rightarrow \bM_{m,n}(\IF)$ is  a
 scalar multiple of a linear isometry for the spectral norm.
In this case,
 $T$ has the form $A \mapsto \gamma UAV$,  or   the form $A \mapsto  \gamma UA^{\TT}V$ (only possible when $m = n$),
for some  $U\in \bU_m(\IF)$ and $V\in\bU_n(\IF)$, and $\gamma\geq 0$; see, e.g., \cite{Li}.
Here,   $\bU_k(\IF)$, or simply $\bU_k$, denotes the group of $k\times k$
complex unitary matrices or real orthogonal matrices
depending on $\IF = \IC$ or $\IR$.
Indeed,
these are merely all linear preservers of TEA/parallel pairs on $\bM_{m,n}(\IF)$.
More precisely, we have the following.

\begin{theorem} \label{main}
Suppose  $m,n \ge 2$, $(m,n) \ne (2,2)$,
and $T:\bM_{m,n}(\IF)\rightarrow \bM_{m,n}(\IF)$  is a nonzero  linear map.
 The following conditions are equivalent.
\begin{enumerate}[{\rm (a)}]
\item $T$ preserves TEA pairs.
\item $T$ preserves parallel pairs  and the range space of $T$ has dimension larger than one.
\item There are $U \in \bU_m(\IF)$, $V \in \bU_n(\IF)$, and $\gamma > 0$ such that

\medskip\hspace{0pt}{
{\rm (1)} \ $T$ has the form $A \mapsto \gamma UAV$, \ \ or \ \
{\rm (2)} \ $m = n$ and $T$ has the form $A \mapsto \gamma UA^{\TT}V$.
}
\end{enumerate}
\end{theorem}

If the range space of a linear map $T: \bM_{m,n}(\IF) \rightarrow \bM_{m,n}(\IF)$ has dimension
at most $1$, or equivalently,
$T$ has the form

\medskip\hspace{20pt}{
{\rm (3)} \
$A \mapsto \tr(FA) Z$ \quad for  some  $F\in \bM_{n,m}(\IF)$ and $Z\in \bM_{m,n}(\IF)$,
}

\noindent
then clearly $T$ preserves parallel pairs (but not necessarily TEA pairs).

The proof of Theorem \ref{main} will be presented in
Section  \ref{S:main-proof} after some preliminary results are
given in Section \ref{S:pre}.
A key step of the proof is to show that
any nonzero linear TEA pair preserver, and any linear parallel pair preserver with range space
of dimension larger than one,  is invertible (see Proposition \ref{prop:tea-para-not-inv}).
Then one can focus on invertible preservers.

In the study of linear preserver problems
on matrices, the $2\times 2$ case often involves special maps and the proofs require special care.
One reason is that there are many
special subgroups and semi-groups of the linear maps on $\bM_2(\IF)$.
In our study, we  experienced the same problem.  Nevertheless, with intricate computation
and insights on some such special subgroups, we are able to give a complete
description of linear maps $T: \bM_2(\IF) \rightarrow \bM_2(\IF)$ preserving TEA/parallel pairs.
It turns out that the structure are different for the two classes of preservers,
and they are also different in the real and complex cases.
To describe our main result, let
$$
I_2 = \begin{pmatrix} 1 & 0 \cr 0 & 1\cr\end{pmatrix}, \
\mathcal{C}_1 = \mathcal{R}_1 = \begin{pmatrix} 1 & 0 \cr 0 & -1\cr\end{pmatrix}, \
\mathcal{C}_2 = \mathcal{R}_2 = \begin{pmatrix} 0 & 1 \cr 1 & 0\cr\end{pmatrix}, \
\mathcal{C}_3 = \begin{pmatrix} 0 & -i \cr i & 0 \cr\end{pmatrix}, \
\mathcal{R}_3 = \begin{pmatrix} 0 & 1 \cr -1 & 0 \cr\end{pmatrix}.
$$
It is well-known that  $\bC = \{I_2/\sqrt 2, \mathcal{C}_1/\sqrt 2, \mathcal{C}_2/\sqrt 2, \mathcal{C}_3/\sqrt 2\}$
is an orthonormal basis of $\bM_2(\IC)$
 with respect to the inner product
$\langle A,B\rangle=\operatorname{trace}(B^*A)$, while
 $\bR=\{I_2/\sqrt 2, \mathcal{R}_1/\sqrt 2, \mathcal{R}_2/\sqrt 2, \mathcal{R}_3/\sqrt 2\}$
is an orthonormal basis of $\bM_2(\IR)$  with respect to the inner product $\langle A, B\rangle = \operatorname{trace}(B^{\TT} A)$.

\begin{theorem}\label{thm-parallel-C}
Let $T: \bM_2(\IC) \rightarrow \bM_2(\IC)$ be a  complex linear map.
\begin{itemize}
\item[(a)]$T$ preserves TEA pairs if and only if there are $\gamma \geq 0$ and
$U, V \in \bU_2(\IC)$ such that $T$ has the form
$$A \mapsto \gamma UAV \qquad \hbox{ or } \qquad  A\mapsto \gamma UA^{\TT}V.$$
\item[(b)] $T$ preserves parallel pairs if and only if one of the following holds.
\begin{itemize}
\item[(b.1)] $T$ has the form $A \mapsto \tr(FA) Z$ for  some  $F, Z\in \bM_2(\IC)$.
\item[(b.2)] There are complex unitary matrices
 $U_1, U_2, V_1, V_2 \in \bU_2(\IC)$ such that the matrix representation of the linear transformation
  $\tilde{T}$ defined by    $A\mapsto U_1 T(U_2AV_2)V_1$   with respect
to the basis $\bC$ has the form
\begin{equation}\label{eq:(b.2)}
\ \hskip .7 in [\tilde{T}]_{\bC} =\begin{pmatrix}
d_0 & \gamma_1 i & \gamma_2 i & \gamma_3 i \\
0 & d_1 & 0 & 0 \cr
0 & 0 & d_2 & 0 \cr
0 & 0 & 0 & d_3\cr
\end{pmatrix}
\end{equation}
for  real scalars
$d_0, d_1, d_2, d_3, \gamma_1, \gamma_2, \gamma_3$.
\end{itemize}
\end{itemize}
\end{theorem}

\begin{theorem}\label{thm-parallel-R}
Let $T: \bM_2(\IR) \rightarrow \bM_2(\IR)$ be a   real  linear map.
\begin{itemize}
\item[(a)] $T$ preserves TEA pairs if and only if $T$  has the form
\begin{itemize}
\item[(a.1)] There are real orthogonal matrices $F_1, F_2 \in \bU_2(\IR)$ with $F_1^{\,\TT}F_2 = \begin{pmatrix} 0 & 1\\ -1 & 0\end{pmatrix}$
and  $Z_1, Z_2\in \bM_2(\IR)$  such that  $T$
has the form
$$X \mapsto \tr(F_1X)Z_1 + \tr(F_2 X)Z_2.$$
\item[(a.2)] There are real orthogonal matrices $U_1, U_2, V_1, V_2 \in \bU_2(\IR)$
such that the matrix representation of the linear transformation
  $\tilde{T}$ defined by  $A\mapsto U_1 T(U_2AV_2)V_1$ with respect to the basis $\bR$ has the diagonal form
\begin{align}\label{eq:(a.2)}
[\tilde T]_\bR = \begin{pmatrix} d_0 & & &\\ & d_1 & &\\ & & d_2 &\\ & & & d_3\end{pmatrix}\in \bM_4(\IR).
\end{align}
\end{itemize}
\item[(b)] $T$ preserves parallel pairs if and only if $T$   has the form in (a.1), (a.2), or
$$
A\mapsto \tr(FA)Z \quad\text{for some  $F, Z \in \bM_2(\IR)$.}
$$
\end{itemize}
\end{theorem}

In Section \ref{S:2x2}, we will present the proofs of Theorems \ref{thm-parallel-C} and \ref{thm-parallel-R}
for the $2\times 2$ case.
The proofs are technical and lengthy. While the techniques and insights may be useful for future study, it would be nice
to have some shorter proofs for the results.
Related results and open problems are mentioned in Section \ref{S:future}.

\section{Preliminaries}\label{S:pre}

Let $m,n \geq 2$, and let $\IF=\IR$ or $\IC$.
For notation simplicity,  let  $\bM_{m,n}$, $\bM_n$ and $\bU_n$ denote
$\bM_{m,n}(\IF)$, $\bM_n(\IF)$ and $\bU_n(\IF)$, respectively, if the discussion applies to
both the real and complex cases. Denote by
$A^{\TT}$ and $A^*$ the transpose and the conjugate transpose of a matrix $A$.
We have $A^* = A^{\TT}$ if $\IF = \IR$.  Let   $\{E_{11}, E_{12}, \dots, E_{mn}\}$ denote the standard basis
for $\bM_{m,n}$, i.e.,  $E_{ij} \in \bM_{m,n}$ has the $(i,j)$ entry equal to 1 and all other entries equal to 0.
Let $I_n \in \bM_n$ be the identity matrix, and  $\mathbf{0}_{r,s}\in \bM_{r,s}$ be the zero matrix.
We simply write $I$ and  $\mathbf{0}$ if the size is clear in the context.
We   write $R\oplus S$ for the
$(r_1 + s_1)\times (r_2 + s_2)$
block matrix $\begin{pmatrix} R & 0\\ 0 & S\end{pmatrix}$ for   $R\in \bM_{r_1\times r_2}$
and $S\in \bM_{s_1\times s_2}$.
Two matrices  $A, B\in \bM_{m,n}$ are
said to be disjoint (a.k.a,\ orthogonal) if  $A^*B=\mathbf{0}\in \bM_n$ and $AB^*=\mathbf{0}\in \bM_m$. In some cases, we also write
$A \oplus B$ for the sum of  disjoint  matrices $A$ and $B$. 

Recall that every
rectangular matrix $A\in \bM_{m,n}$ admits a singular value decomposition, namely,
\begin{equation}\label{eq:SVD}
A =  U\begin{pmatrix} D & 0 \cr 0 & 0\cr\end{pmatrix} V^*= \sum_{j=1}^k s_j u_j v_j^*,
\end{equation}
where $D = \diag(s_1, \dots, s_k)$ with $s_1 \ge \cdots \ge s_k > 0$, the nonzero singular values
of $A$,  and
$u_1, \dots, u_k$ and $v_1, \dots, v_k$ are the first $k$ columns of the matrices $U \in \bU_m$ and
$V \in \bU_n$.
Note that $k \le \min\{m,n\}$ and $s_1^2, \dots, s_k^2$ are the nonzero eigenvalues of $AA^*$ (or $A^*A$).
We will let $s_j(A) = s_j$ for $j = 1, \dots, k$,
and $s_j(A) = 0$ for $j > k$.

It is easy to see that two nonzero rectangular matrices $A, B\in \bM_{m,n}$ are parallel if and only if there is
a unit vector $v\in \IF^n$ and $\mu\in \IF$ with $|\mu| = 1$ such that
$$
\|A + \mu B\| = \|Av + \mu Bv\| = \|A\| + \|B\| \geq \|Av\| + \|\mu Bv\|\geq \|Av + \mu Bv\|.
$$
This is equivalent to the condition
$$
|\langle Av, Bv\rangle| = \|Av\|\|Bv\|=\|A\|\|B\| > 0.
$$
Then the Cauchy-Schwartz inequality theorem tells us that
$Av, Bv$ are linearly dependent.  Let $u= Av/\|Av\|$.  Then $A = s_1(A) uv^* \oplus A_1$ and
$B = \overline{\mu} s_1(B) uv^* \oplus B_1$ for some unimodular scalar $\mu$ and matrices $A_1, B_1$
with $\|A_1\|\leq \|A\|=s_1(A)$ and $\|B_1\|\leq \|B\|=s_1(B)$.
This leads to the following lemma; see, e.g., \cite[Theorem 3.1 and Corollary 3.2]{Li0}.

\begin{lemma}\label{lem:paralle-cond}
Let $A, B\in \bM_{m,n}$ be nonzero rectangular matrices.
\begin{enumerate}[{\rm (a)}]
  \item $A, B$ are parallel if and only if there is a unit vector $v\in \IF^n$ such that
  $$
  |\langle Av, Bv\rangle| = \|A\|\|B\|.
  $$
  In this case, $A,B$ form a TEA pair  when $\langle Av, Bv\rangle = \|A\|\|B\|$.

    \item $A, B$ are parallel if and only if there are unitary matrices $U\in \bU_m$, $V\in \bV_n$ such that
  $$
  A = U\begin{pmatrix} a & 0\\ 0  &  A_1\end{pmatrix}V \quad\text{and}\quad
  B =  U\begin{pmatrix} b & 0\\ 0  &  B_1\end{pmatrix}V
  $$
  for some   $A_1, B_1\in \bM_{m-1, n-1}$ with $\|A_1\|\leq a$
   and $\|B_1\|\leq |b|$.  In this case, $A, B$ form a TEA pair  when $b>0$.

  \item If $A, B$ are parallel (resp.\ TEA), then so are $\alpha A, \beta B$ for any scalars $\alpha, \beta$ (resp.\ such that $\overline{\alpha}\beta \geq 0$).

  \item $A=U\begin{pmatrix} A_1 & 0\\0 & A_2\end{pmatrix}V$ and $B=U\begin{pmatrix} B_1& 0 \\ 0& B_2\end{pmatrix}V$ 
  are not parallel if  $U\in \bU_m$, $V\in \bV_n$, 
  $A_1, B_1\in \bM_{r,s}$ and $A_2, B_2\in \bM_{m-r,n-s}$   with 
  $\|A_2\|< \|A_1\|$ and $\|B_1\|<\|B_2\|$.
\end{enumerate}
\end{lemma}

It follows from Lemma \ref{lem:paralle-cond}(d) that two nonzero rectangular matrices
$A, B$ can not be simultaneously  parallel and orthogonal.

\begin{lemma}\label{lem:limit}
\begin{enumerate}[{\rm (a)}]
  \item Let $T_k: \bM_{m,n}\to \bM_{m,n}$ be a sequence of linear maps preserving parallel pairs (resp.\ TEA pairs) for $k\geq 1$.
  If $T_k \to T$ in norm then $T$ preserves parallel pairs (resp.\ TEA pairs).
  \item If $(A_k, B_k)$ are parallel pairs (resp.\ TEA pairs)
  and $A_k\to A$, $B_k\to B$ in norm, then $(A, B)$ is a parallel pair (resp.\ TEA pair).
\end{enumerate}
\end{lemma}
\begin{proof}
  (a) Let $A,B\in \bM_{m,n}$ be parallel, and so are $T_k(A), T_k(B)$ for $k\geq 1$.
  In other words, there is a unimodular scalar $\mu_k$ such that
  $$
  \|T_k(A)+\mu_k T_k(B)\| = \|T_k(A)\| + \|T_k(B)\|, \quad\forall k\geq 1.
  $$
  Since the unit circle in $\IF$ is compact, we can assume $\mu_k\to \mu$ in $\IF$ with $|\mu|=1$.
  Letting $k\to\infty$, we have
  $$
  \|T(A) + \mu T(B)\| = \|T(A)\| + \|T(B)\|.
  $$
  Consequently, $T$ preserves parallel pairs.  The case for TEA pair preservers is similar.

  (b) The proof  is similar to (a).
\end{proof}

\begin{example}\label{eg:parallel}
\begin{enumerate}[{\rm (a)}]
    \item Two linearly dependent matrices  are always parallel.

      \item Any two complex unitary matrices $U,V\in \bU_n(\IC)$ are parallel.  Indeed,
  $$\| U + \mu V\| = \|I +\mu U^*V\| = 2 = \|U\| + \|V\|$$
  whenever the complex conjugate $\overline{\mu}$ belongs to the spectrum $\sigma(U^*V)$ of the complex unitary $U^*V$.
  They form a TEA pair  when $\sigma(U^* V)$ contains $1$.

  \item Two real orthogonal matrices $U,V\in \bU_n(\IR)$ are parallel if any only if $\sigma(U^{\TT} V)$ contains $1$ or $-1$.
  They form a TEA pair  when $\sigma(U^{\TT} V)$ contains $1$.

  \item Two real orthogonal matrices $F_1, F_2\in \bU_2(\IR)$ are parallel exactly when $F_1  = \pm F_2$, or
  $\det(F_1)= -\det(F_2)$.  In this case, they form a TEA pair unless $F_1=-F_2$.

\end{enumerate}
\end{example}

\begin{lemma} \label{lem:new}
Let $A, B \in \bM_{m,n}$  such that $\|A\| = \|B\| = 1 =  \|A+B\|/2$.
Then $A+B$ and $A-B$ are parallel if and only if $A =  B$.
\end{lemma}

\begin{proof} \rm If $A = B$, then clearly $A+B=2A$ and $A-B = 0$ are parallel.
Conversely, suppose $A+B$ and $A-B$ are parallel.
Since $\|A+B\| = 2$, there are $U \in \bU_m, V \in \bV_n$ such that
$U(A+B)V = (2I_k) \oplus C$ with $\|C\| < 2$ and $1\leq k \leq \min\{m,n\}$.
Since $\|A\| = \|B\| = 1$,  the $(j,j)$ entry of $UAV$ and $UBV$ must be 1
for $j = 1, \dots, k$.
Consequently,  $UAV = I_k \oplus A_1$, $UBV = I_k \oplus B_1$,
and thus $U(A-B)V=0\oplus (A_1-B_1)$.
With  $\|C\|<2$,
it follows from Lemma \ref{lem:paralle-cond}(d)  that
$A_1-B_1=0$.  Thus, $A=B$.
\end{proof}

\begin{lemma} \label{lem:par-1dim}
If $\bV$ is a subspace of $\bM_{m, n}$ such that any two elements in $\bV$ are parallel,
then $\dim \bV \le 1$.
\end{lemma}

\begin{proof}
Let $A, B\in \bV$ with $\|A\|=\|B\|=1$.
Since $A$ and $B$ are parallel, $\|A+\mu B\|=\|A\|+\|B\|=\|A\|+\|\mu B\|=2$ for some unimodular scalar $\mu$.
Now, $A+\mu B, A-\mu B \in \bV$ are parallel. By Lemma \ref{lem:new},
$A = \mu B$.  Since  this is true for any norm one matrices $A, B \in \bV$, we see that
$\bV$ has dimension at most one.
\end{proof}

Recall that for $p \ge 1$,
the $\ell_\infty$-norm on $\IF^m$ is defined as:
$\|x\|_{\ell_p} = (\sum_{j=1}^m |x_j|^p)^{1/p}$, and
$\|x\|_{\ell_\infty}  = \max_j |x_j|$
for $x = (x_1, \dots, x_m)^{\TT}\in \IF^m$.

\begin{lemma}[{\cite[Lemma 2.4]{LTWW-LP}}] \label{D-matrix} Suppose $B = (b_{rs}) \in \bM_m$
such that either  $b_{jj} > |b_{kj}|$ whenever $j \ne k$, or $b_{jj} > |b_{jk}|$ whenever $j \ne k$.
The following conditions are equivalent.
\begin{itemize}
\item [{\rm (a)}]
For any
 $x = (x_1, \dots, x_m)^{\TT}\in \IF^m$ with $y = Bx = (y_1, \dots, y_m)^{\TT}$, we have
$$
\|x\|_{\ell_\infty} = |x_r| > |x_s|\ \text{whenever  $s \ne r$}\quad
\implies\quad \|y\|_{\ell_\infty}= |y_r| \geq |y_s|\ \text{whenever $s\ne r$.}
$$
\item[{\rm (b)}] Either $m = 2$ and $B = B^*$ with $b_{11}=b_{22}>|b_{12}|$, or  $m > 2$ and $B = b_{11}I_m$.
\end{itemize}
\end{lemma}
\begin{proof}
We sketch a proof for completeness.
It suffices to verify the implication (a) $\Rightarrow$ (b), as the other implication can be verified directly.
Replacing $B$ by $P^\TT BP$ with a suitable permutation matrix $P$, we can assume that
 the first row $v_1= \begin{pmatrix} b_{11}& b_{12}& \cdots & b_{1m}\end{pmatrix}$
  of $B$ has   the maximal $\ell_1$ norm.  Replacing  $B$ by $U^*B U$ with a suitable diagonal  $U \in \bU_m$,
we can further assume all $b_{1j}\geq 0$.

By the continuity and an induction argument,  for $x = (x_1, \dots, x_m)^{\TT}$, $y = B x = (y_1, \dots, y_m)^{\TT}$
and $k=1,\ldots, m$, we have
\begin{align}\label{eq:kequal}
|x_1| = \cdots = |x_k|> |x_r|\ \text{for all $r>k$}\quad\implies\quad
 |y_1| = \cdots = |y_k|\geq |y_r| \ \text{for all $r>k$.}
\end{align}
Let  $x = (1,\dots, 1)^{\TT}$ and ${y}=Bx$. If $v_1, \dots, v_m$ are the rows of
$B$, then
\begin{align*}
\|v_{1}\|_{\ell_1} &= b_{11} + \cdots + b_{1m} =
|y_{1}| = |y_j| =
|b_{j1} + \cdots + b_{jm}| \\
&\le |b_{j1}| + \cdots + |b_{jm}|
= \|v_j\|_{\ell_1} \le \|v_1\|_{\ell_1}\quad\text{for $j \geq 1$.}
\end{align*}
Since $b_{jj} > 0$, we see that $b_{ij} \ge 0$ for all $i, j$,
and all row sums of $D$ are equal, say, to $s>0$.

Taking $x = (1,1, 0,\ldots,0)^{\TT}$, with \eqref{eq:kequal} we have $b_{11}+b_{12}=b_{21}+b_{22}$.
Similarly, taking $x = (1,-1, 0,\ldots,0)^{\TT}$, we have $|b_{11}- b_{12}| = |b_{21} - b_{22}|$.
It follows from either the assumption  $b_{11}> b_{21}$ and $b_{22}> b_{12}$, or the assumption
 $b_{11}> b_{12}$ and $b_{22}> b_{21}$ that
$b_{11} - b_{12} =  b_{22}-b_{21}$.
Consequently, $b_{11} = b_{22}$
and $b_{21} = b_{12}$.
The assertion follows when $m=2$.

Suppose $m\geq 3$.
Arguing similarly,
we see that
$$
b_{11}=\cdots = b_{mm}\quad\text{and}\quad b_{jk}=b_{kj}\quad\text{whenever $j\neq k$.}
$$
For a fixed $j = 1, \dots, m$, we take $x = (1,\ldots, 1, \underbrace{-1}_{j\text{th}}, 1,\ldots, 1)^{\TT}$
and ${y}=Bx$.
For   distinct indices $j,k, l$, we have
$|y_{l}| = s -2 b_{lj}=|y_k| = s - 2b_{kj}$.
It follows $b_{lj} = b_{kj}=b_{jk}$.
Thus,
$b_{jk}=b_{12}$ are all equal for $j \ne k$.
Consider $u = (1, -1, \ldots, -1)^{\TT}$ and ${v}=Bu= (v_1, v_2,  \ldots, v_m)^{\TT}$.
Then $|v_1|=|v_2|$ implies either
$$
b_{11} -  (m-1)b_{12}=  b_{11} + (m-3)  b_{12}\quad\text{or}\quad (m-1)b_{12}-b_{11}= b_{11} + (m-3)b_{12}.
$$
  Since $m\geq 3$,  either
$b_{12}=0$ or $b_{11}=b_{12}$.  But $b_{11}> b_{12}$.  This implies that $B=b_{11}I_m$.
\end{proof}

\section{The case $(m,n)\neq (2,2)$}\label{S:main-proof}


This section is devoted to proving
Theorem \ref{main}, for which we assume $(m,n) \ne (2,2)$.
We will assume that $n\geq m \geq 2$.
The proof for the case when $2\leq n \le m$ can be done similarly.
The proof of Theorem \ref{main} is quite involved.
With the following observation, we can focus on those invertible preservers.
The proof of Proposition \ref{prop:tea-para-not-inv} will be given in subsection \ref{SS:invertibility}.

\begin{proposition}\label{prop:tea-para-not-inv}
   Suppose $m,n \ge 2$.
Let $T: \bM_{m,n}(\IF) \rightarrow \bM_{m,n}(\IF)$ be a nonzero linear map. Then $T$ is invertible if one of the following holds.
\begin{itemize}
\item[(a)] $\bM_{m,n}(\IF)\neq \bM_{2}(\IR)$ and $T$ preserves TEA pairs.
\item[(b)]  $(m,n)\neq (2,2)$ and $T$ preserves parallel pairs with range space of dimension larger than one.
\end{itemize}
\end{proposition}

\subsection{Proof of Theorem \ref{main}}

The implication (c) $\Rightarrow$ (a) in Theorem \ref{main} is clear.
If (a) holds, then  $T$ is invertible by Proposition \ref{prop:tea-para-not-inv}.
Clearly, $T$   preserves parallel pairs and the range space of $T$ has dimension larger than 1.
So, condition (b) follows. Suppose (b) holds. By Proposition \ref{prop:tea-para-not-inv}, $T$ is invertible.
We will show that $T$ has one of the forms in (c).

For $A \in \bM_{m,n}$, let
$$
\mathcal{P}(A) = \{B\in \bM_{m,n}: B \text{ is parallel to }  A\}.
$$
Recall that $s_k(A)$ denotes the $k$th singular value of $A$ in descending order.

\begin{lemma} \label{ck1}
Let $A \in \bM_{m,n}$.
\begin{itemize}
\item[{\rm (a)}]
If $s_1(A) = s_2(A)$, then  {$\dim(\Span \mathcal{P}(A))\geq mn-1$}.
\item[{\rm (b)}]
Suppose $s_1(A) > s_2(A)$
and   $U \in \bU_m, V \in \bU_n$
such that $A = U([a] \oplus F)V^*$ with $\|F\|< |a|$.
Then
\begin{align}\label{eq:PA-g}
\mathcal{P}(A) =  \{U ([g] \oplus G)V^*: G \in \bM_{m-1,n-1}, |g| \ge \|G\|\}.
\end{align}
In particular, $\Span(\mathcal{P}(A))$ has dimension $(m-1)(n-1)+1$.
\item[{\rm (c)}]
If  $\Span(\mathcal{P}(A))$ has dimension $(m-1)(n-1) + 1$,
then $s_1(A) > s_2(A)$ so that \eqref{eq:PA-g} holds.

\end{itemize}

\end{lemma}

\begin{proof}
(a) If $s_1(A) = s_2(A)$, we may assume that
$A = I_2
\oplus Z$ for some $Z \in \bM_{m-2,n-2}$ with $\|Z\| \le 1=\|A\|$.
Then $\mathcal{P}(A)$ contains matrices of the form
$[b] \oplus X$  with $X\in \bM_{m-1,n-1}$
with $\|X\|\leq |b|$.
It follows that  $B=(b_{ij})\in \Span(\mathcal{P}(A))$ whenever $b_{1j}=b_{i1}=0$
for all $i,j \neq 1$.
Similarly, we also have $B=(b_{ij})\in \Span(\mathcal{P}(A))$ whenever $b_{2j}=b_{i2}=0$ for all $i,j \neq 2$.
Furthermore, $E_{12} + E_{21} \in \mathcal{P}(A)$.
It also holds that $ i(E_{12} - E_{21}) \in \mathcal{P}(A)$ if $\mathbb{F}=\mathbb{C}$. Thus,
$\mathcal{P}(A)$ spans $\bM_{m,n}$ if $\mathbb{F}=\mathbb{C}$, and $\dim (\Span \mathcal{P}(A))\geq mn-1$ if $\mathbb{F}=\mathbb{R}$.

(b) Assume $s_1(A) > s_2(A)$.
Let $U \in \bU_m$ and $V \in \bU_n$ satisfy
$A = U([a] \oplus F)V^*$ with $\|F\| < a$.
Clearly,
$$\mathcal{P}(A) = \{U([g] \oplus G)V^*: G \in \bM_{m-1,n-1}, \| G \| \leq |g|\}.$$
So, $\Span(\mathcal{P}(A))$ has dimension $(m-1)(n-1)+1$.

(c) If $\dim \Span(\mathcal{P}(A)) = (m-1)(n-1)+1$, then
$s_1(A) \ne s_2(A)$ by (a). The result follows.
\end{proof}

In the following, we  assume that $T$ is an invertible linear parallel pair preserver.

\begin{lemma} \label{Assertion 1}
 If $X \in \mathcal{P}(A)$, then $T(X) \in \mathcal{P}(T(A))$.  Moreover,
$$T(\mathcal{P}(A)) \subseteq \mathcal{P}(T(A)) \qquad \hbox{ and } \qquad
T(\Span \mathcal{P}(A)) \subseteq \Span \mathcal{P}(T(A)).$$
\end{lemma}

\begin{proof} This follows readily from the definition. \end{proof}

\begin{lemma} \label{Assertion 2}
 If $s_1(T(A)) > s_2(T(A))$ then $s_1(A) > s_2(A)$.
In this case, there are $U, X \in \bU_m$, and $V, Y \in \bU_n$ such that
\begin{align*}
A &= U([a] \oplus F)V \hbox{ with } a > \|F\|, \quad
\Span \mathcal{P}(A) = \{ U([z]\oplus Z)V: z \in \mathbb{F}, Z\in \bM_{m-1,n-1}\},\\
T(A) &= X([b] \oplus G)Y \hbox{ with } b > \|G\|, \quad
\Span \mathcal{P}(T(A)) = \{ X([z]\oplus Z)Y: z \in \mathbb{F}, Z \in \bM_{m-1,n-1}\}.
\end{align*}
We have $T(\Span \mathcal{P}(A)) = \Span \mathcal{P}(T(A))$, i.e., $T$ is a
linear bijection between the two subspaces $\Span \mathcal{P}(A)$ and $\Span \mathcal{P}(T(A))$.
\end{lemma}

\begin{proof}
If $s_1(T(A)) > s_2(T(A))$, then $\Span \mathcal{P}(T(A))$ has dimension
$(m-1)(n-1) + 1$.  By Lemma \ref{ck1}, if $s_1(A) = s_2(A)$,
 then
 $\dim(T(\Span \mathcal{P}(A))) = \dim(\Span \mathcal{P}(A))\geq mn-1$.
By Lemma \ref{Assertion 1}, $T(\Span \mathcal{P}(A)) \subseteq \Span(\mathcal{P}(T(A)))$,
and the latter has dimension $(m-1)(n-1) +1$.  This is impossible.
Therefore, $s_1(A) > s_2(A)$
and $T(\Span \mathcal{P}(A))$ and $\Span \mathcal{P}(A)$ have the same dimension, and thus
$T(\Span \mathcal{P}(A)) = \Span \mathcal{P}(T(A))$.
The description of $\mathcal{P}(A)$ and $\mathcal{P}(T(A))$ are now clear.
\end{proof}

\begin{lemma} \label{Assertion 3}
There are $U \in \bU_m$, $V \in \bU_n$, and $D=(d_{ij})\in \bM_m$ such that
$$
T^{-1}(E_{jj}) =  U(\sum_{{k}=1}^m  d_{j{k}} E_{{k},{k}})V, \quad \text{for $j=1,\ldots, m$,}
$$
and    $d_{jj} > |d_{j{k}}|$ for all  $j,k = 1, \dots, m$ with $j\neq k$.
\end{lemma}

\begin{proof}
By Lemma \ref{Assertion 2}, if $T(A_1) = E_{11}$, then there are
$U\in \bU_m$ and $V\in \bU_n$ such that
\begin{gather*}
A_1 = U([d_{11}] \oplus F)V\quad\text{with}\quad d_{11}> \|F\|, \\
\Span(\mathcal{P}(A_1)) = \{U([x] \oplus X)V: x\in \mathbb{F}, X \in \bM_{m-1,n-1}\},
\intertext{and}
 T(\Span \mathcal{P}(A_1)) = \Span \mathcal{P}(E_{11}) = \{ [y] \oplus Y: y \in \mathbb{F},
Y\in \bM_{m-1,n-1}\}.
\end{gather*}
Now, suppose $T(A_2) = E_{22}$. Since $E_{22} \in \Span \mathcal{P}(E_{11})$, we see that
$A_2 \in \Span \mathcal{P}(A_1)$ and $A_2=U([d_{21}] \oplus B)V$ with
$B\in \bM_{m-1,n-1}$. Since $T(A_1) = E_{11}$ and $T(A_2) = E_{22}$
are not parallel, $A_1$ and $A_2$ are not parallel.
So, $\|B\| > |d_{21}|$, and $s_1(B)>s_2(B)$.  By Lemma \ref{Assertion 2},
we may change  $U, V$ to
$U([1]\oplus U_1)$ and $([1]\oplus V_1)V$ for suitable $U_1 \in \bU_{m-1},
V_1 \in \bU_{n-1}$
and assume that $A_1 = U([d_{11}] \oplus F_1)V$
and $A_2 = U([d_{21}] \oplus [d_{22}] \oplus F_2)V$ such that
$d_{22} > |d_{21}|$ and $d_{22} > \|F_2\|$.
 By Lemma \ref{Assertion 2} again,
$$
\Span(\mathcal{P}(A_2))
= \{U(x_{ij})V: x_{22} \hbox{ is the only nonzero entry
in row 2 and column 2} \},$$
and
\begin{align*}
T(\Span(\mathcal{P}(A_2))) &= \Span(\mathcal{P}(E_{22}))\\
 &= \{ (y_{ij}):
y_{22}   \hbox{ is the only nonzero entry
in row 2 and column 2} \}.
\end{align*}
Since $E_{11} \in \Span \mathcal{P}(E_{22})$, we see that
$A_1 \in \Span \mathcal{P}(A_2)$. Thus,
$A_1 = U([d_{11}] \oplus [d_{12}] \oplus \hat F_1)V$, with $\|\hat F_1\|< \|A_1\|=d_{11}$ and $|d_{12}|< d_{11}$.

Now, suppose $m\geq 3$ and $T(A_3) = E_{33}$. Since $E_{33} \in \Span \mathcal{P}(E_{jj})$
for $j = 1, 2$, we have
$A_3 \in \Span \mathcal{P}(A_j)$ for $j = 1,2$. It follows that
$A_3 = U(\diag(d_{31}, d_{32}) \oplus F_3)V$.
Since $T(A_3) = E_{33}$ is not parallel to $T(A_j) = E_{jj}$ for $j = 1,2$,
we see that $\|A_3\| = \|F_3\|> |d_{3j}|$ for $j = 1,2$.
By Lemma \ref{Assertion 2}, we may change  $U, V$ to
$U(I_2\oplus U_2)$ and $(I_2\oplus V_2)V$ for suitable $U_2 \in \bU_{m-2},
V_2 \in \bU_{n-2}$
and assume that $A_1 = U(\diag(d_{11}, d_{12}) \oplus G_1)V$,
$A_2 = U(\diag(d_{21}, d_{22}) \oplus G_2)V$
and $A_3 = U (\diag(d_{31}, d_{32}, d_{33}) \oplus G_3)V$ with $d_{33}> |d_{31}|, |d_{32}|, \|G_3\|$.
By Lemma \ref{Assertion 2} again,
$$
\Span(\mathcal{P}(A_3)) = \{U(x_{ij})V: x_{33} \in \mathbb{F} \hbox{ is the only nonzero entry
in row 3 and column 3} \},$$
and
\begin{align*}
T(\Span(\mathcal{P}(A_3))) &= \Span(\mathcal{P}(E_{33}))\\
&= \{ (y_{ij}):
y_{33} \in\mathbb{F} \hbox{ is the only nonzero entry
in row 3 and column 3} \}.
\end{align*}
Since $E_{11}, E_{22} \in \Span \mathcal{P}(E_{33})$, we see that
$A_1, A_2 \in \Span \mathcal{P}(A_3)$. Thus,
\begin{align*}
A_1 &= U(\diag(d_{11}, d_{12}, d_{13}) \oplus \hat G_1)V,
\quad\text{and} \\
A_2 &= U(\diag(d_{21}, d_{22}, d_{23}) \oplus \hat G_2)V.
\end{align*}
Neither of $T(A_1), T(A_2)$ is parallel to $T(A_3)$,  and thus neither of $A_1, A_2$ is parallel to $A_3$.
Consequently, $d_{11}> |d_{13}|$ and $d_{22}> |d_{23}|$.

We can repeat the argument for $A_4, \dots, A_m$,
and modify $U, V$ in each step until we get the conclusion that
$$
A_j = U(\diag(d_{j1}, d_{j2}, \ldots, d_{jm}))V\quad\text{with}\quad
d_{jj}> |d_{jk}|, \quad\text{whenever}\ j\neq k,\ \text{and}\ j,k=1,\ldots,m.
$$
\end{proof}

\begin{lemma}\label{lem:diagonal}
\begin{enumerate}[{\rm (a)}]
  \item When $m\geq 3$, there are unitary matrices $U\in \bU_m$, $V\in \bU_n$, and scalar  $d_{11}>0$ such that
  $$
  T^{-1}(E_{jj}) = U(d_{11} E_{jj})V \quad\text{for $j=1,\ldots, m$.}
  $$
  \item  When $m=2$, there are unitary matrices $U\in \bU_2$, $V\in \bU_n$,
  and scalars $d_{11}$ and $d_{12}$ with $d_{11}>|d_{12}|$ such that
\begin{align*}
  T^{-1}(E_{11}) &= U(d_{11} E_{11} +d_{12} E_{22})V, \\
  T^{-1}(E_{22}) &= U( \overline{d_{12}} E_{11} + d_{11} E_{22})V.
\end{align*}
\end{enumerate}
\end{lemma}
\begin{proof}
We continue the argument and notation usage in the proof of Lemma \ref{Assertion 3}.
For notation simplicity, we assume that
$U = I_m, V = I_n$.
Recall that
\begin{align*}
T^{-1}(x_1 E_{11} + \cdots + x_mE_{mm})
&= x_1(\sum_{k=1}^{m} d_{1k} E_{kk}) + \cdots + x_m(\sum_{k=1}^{m} d_{mk} E_{kk})\\
 &= y_1 E_{11} + \cdots + y_m E_{mm}
\end{align*}
with $(y_1, \dots, y_m)^{\TT} = D^{\TT}(x_1, \dots, x_m)^{\TT}$,
where $D=(d_{ij})\in \bM_m$ with $d_{jj} > |d_{j{k}}|$ whenever $j\neq k$.

When $m=2$, since $T$ preserves parallel pairs, the $2\times 2$ matrix $B=D^\TT$
satisfies the condition (a) in  Lemma \ref{D-matrix} by Lemma \ref{Assertion 2},
and thus  $D=D^*$ with $d_{11}=d_{22}> |d_{12}|$.
The assertion follows.

Suppose $m\geq 3$.
If $X = \sum_{j=1}^m x_j E_{jj}$
with $|x_{l}| > |x_i|$ for all $i \ne {l}$, then
$X$ is parallel to $E_{{l}{l}}$ but not
any other $E_{ii}$ with $i\ne {l}$.  Consequently,
$$
Y = T^{-1}(X) =\sum_{j=1}^m x_j A_j = \sum_{j=1}^m y_j E_{jj}
$$
is not parallel to $T^{-1}(E_{ii}) = A_i$, and thus not parallel to  $E_{ii}$
for any $i \ne {l}$. It forces $Y$  parallel to $E_{{l}{l}}$ but not any other $E_{ii}$.
As a result, $|y_{l}| > |y_i|$ for all $i\ne {l}$.
Thus, the matrix $B=D^\TT$ satisfies the condition (a) in
 Lemma \ref{D-matrix}. It follows $D=d_{11}I_m$, and the assertion follows.
\end{proof}



\begin{proof}[Proof of Theorem \ref{main}]
It remains to verify the implication $(b) \Rightarrow (c)$ in
Theorem \ref{main}.
We verify that $T^{-1}$ sends disjoint rank one matrices to disjoint matrices, and
then $T^{-1}$, as well as $T$, has the asserted form by \cite[Theorem 2.1]{LTWW} (and its remark (iii)).

To see this, let $X_1, X_2\in \bM_{m,n}$ be  disjoint rank one matrices with norm one.
There are $P \in \bU_m, Q \in \bU_n$ such that
$X_j =   PE_{jj}Q$   for $j=1,2$.  Enlarge them with $X_j =   PE_{jj}Q$   for $j=3, \ldots, m$.
Let $S: \bM_{m,n}\to \bM_{m,n}$ be defined by $S(X) = P^*T(X)Q^*$ for $X \in \bM_{m,n}$.
It is clear that the bijective linear map $S$ also preserves parallel pairs, and $S^{-1}(E_{jj})= T^{-1}(X_j)$
  for $j=1,2, \ldots, m$.  We might assume $X_1=E_{11}$ and $X_2=E_{22}$

\medskip
 \noindent
 \textbf{Case 1: $n \ge m \ge 3$.}\
By Lemma \ref{lem:diagonal}, $T^{-1}(X_1), T^{-1}(X_2)$ are disjoint.

\medskip
 \noindent
\textbf{Case 2: $n > m = 2$.} \
By Lemma \ref{lem:diagonal}, there are $U'\in \bU_2$ and $V'\in \bU_n$ such that
\begin{align*}
  T^{-1}(E_{11}) &= U'(d_{11} E_{11} + d_{12} E_{22})V', \\
  T^{-1}(E_{22}) &= U'(\overline{d_{12}} E_{11} + d_{11} E_{22})V'
\end{align*}
 with $d_{11}>|d_{12}|$.
We may further assume that
$$
T^{-1}(E_{11}) =  \,
\begin{bmatrix}
 \begin{matrix}
  1 & 0 \\
  0 & \mu
 \end{matrix} & \rvline &
 \bigzero_{2, n-2}
\end{bmatrix}\quad\text{and}\quad
T^{-1}(E_{22})= \,\begin{bmatrix}
 \begin{matrix}
  \bar \mu  & 0 \\
  0 & 1
 \end{matrix} & \rvline &
 \bigzero_{2, n-2}
\end{bmatrix}
$$
for some scalar $\mu$ with $|\mu| <1$.
Consider the restriction of $T^{-1}$ on the span of
$\{E_{11}, E_{23}\}$.  By Lemma \ref{lem:diagonal} again, we see that $T^{-1}(E_{11}) = U\,
\begin{bmatrix}
 \begin{matrix}
  1 & 0 &0\\
  0 & 0 & \gamma
 \end{matrix} & \rvline &
 \bigzero_{2, n-3}
\end{bmatrix}V$ and
$T^{-1} (E_{23})= U\,
\begin{bmatrix}
\begin{matrix}
  \bar \gamma  & 0 & 0 \\
  0 & 0& 1
 \end{matrix} & \rvline &
 \bigzero_{2, n-3}
\end{bmatrix}V$
for some  $U\in \bU_2, V \in \bU_n$ and $|\gamma| <1$.
Interchanging the second and the third rows of the unitary matrix $V$, we can write
$$
T^{-1}(E_{11})
=  U\,\begin{bmatrix}
 \begin{matrix}
  1 &   0 \\
  0 &   \gamma
 \end{matrix} & \rvline &
 \bigzero_{2, n-2}
\end{bmatrix}V
\quad \hbox{ and } \quad
T^{-1}(E_{23})
=  U\,\begin{bmatrix}
 \begin{matrix}
  \overline{\gamma} &   0 \\
  0 &   1
 \end{matrix} & \rvline &
 \bigzero_{2, n-2}
\end{bmatrix}V.
$$
Observe that
$$ T^{-1}(E_{11}) =\,\begin{bmatrix}
 \begin{matrix}
  1 & 0  \\
  0 & \mu
 \end{matrix} & \rvline &
 \bigzero_{2, n-2}
\end{bmatrix}
=  U\,\begin{bmatrix}
 \begin{matrix}
  1 &  0 \\
  0 &   \gamma
 \end{matrix} & \rvline &
 \bigzero_{2, n-2}
\end{bmatrix}V.
$$
A direct computation shows that
$V=V_1\oplus V_2$ for some   $V_1\in \bU_2$, $V_2\in \bU_{n-2}$ and
$$ \diag(1, \mu) =  U\,\diag(1, \gamma)V_1.
$$
Evaluating the determinants of the matrices on both sides, we get $|\mu| = |\gamma|<1$.
Let $U^*=(u_{ij})$ and $V_1=(v_{ij})$.
Notice that
 the first row vectors of both sides of $U^*\diag(1, \mu)= \diag(1, \gamma)V_1$  have norm one.
Because $U^*$ is a unitary matrix, we have $|u_{11}|=1$ and $u_{12}= 0$, and then $u_{21}= 0$.
Hence $U$ and $V_1$ are unitary diagonal matrices with $u_{11}=v_{11}$ and $\mu u_{22} = \gamma v_{22}$.
If  $|\mu|=|\gamma|\neq0$, then
$$T^{-1} (E_{23})= U\,\begin{bmatrix}
\begin{matrix}
  \bar \gamma  & 0 \\
  0 & 1
 \end{matrix} & \rvline &
 \bigzero_{2, n-2}
\end{bmatrix}V
=
\begin{bmatrix}
\begin{matrix}
  \overline{u_{11}}v_{11}\bar \gamma  & 0 \\
  0 & \overline{u_{22}}v_{22}
 \end{matrix} & \rvline &
 \bigzero_{2, n-2}
\end{bmatrix}
=
\begin{bmatrix}
\begin{matrix}
  \bar \gamma  & 0 \\
  0 & \frac{\mu}{\gamma}
 \end{matrix} & \rvline &
 \bigzero_{2, n-2}
\end{bmatrix}
$$
is parallel to $T^{-1}(E_{22})$, which is impossible because $E_{22}$ is not parallel to
$E_{23}$. Therefore $\mu=\gamma=0$.
In particular, $T^{-1}(E_{11})$ and $T^{-1}(E_{22})$ are disjoint, as asserted.
 \end{proof}

\subsection{Invertibility of parallel/TEA preservers}\label{SS:invertibility}

In this subsection, we prove Proposition \ref{prop:tea-para-not-inv}.
Recall that we can assume $n\geq m\geq 2$ and $(m,n)\neq (2,2)$.

\begin{lemma} \label{new}
Let $Z_1, Z_2\in \bM_{2,n}$ such that $Z_1\neq 0$.
 Suppose
$Z_1 + \xi Z_2$ and $Z_1-\xi Z_2$ attain the triangle
equality for all $\xi \in \IR$. Then $Z_2 =0$.
\end{lemma}

\begin{proof}
We can assume that $Z_1 = d_1 E_{11}+d_2 E_{22}\in \bM_{2,n}$ is nonzero with $d_1 \ge d_2 \ge 0$
and $Z_2 = (\gamma_{ij})\in \bM_{2,n}$.
For $\xi > 0$,
$$2 d_1 = (d_1+\xi \operatorname{Re}\,(\gamma_{11})) + (d_1 - \xi \operatorname{Re}\,(\gamma_{11}))
\le |d_1 + \xi \gamma_{11}| + |d_1 - \xi \gamma_{11}|$$
$$
\qquad \le  \|Z_1+\xi Z_2\|+ \|Z_1-\xi Z_2\| = \|(Z_1 + \xi Z_2)+(Z_1-\xi Z_2)\| = 2d_1.$$
Thus,
$$d_1\pm \xi \operatorname{Re}\,(\gamma_{11}) = |d_1 \pm \xi \gamma_{11}| =  \|Z_1 \pm \xi Z_2\|,$$
and $d_1 \pm \xi \operatorname{Re}\,(\gamma_{11})$ are positive for all $\xi>0$. Hence, $\operatorname{Re}\,(\gamma_{11})=0$, and then $\gamma_{11}=0$.
Moreover, the $(1,1)$ entry of $Z_1 \pm \xi Z_2$ attains the norm of $Z_1 \pm \xi Z_2$.
Hence, the off-diagonal entries of the first row and the first column
of $Z_1 \pm \xi Z_2$ are zeros, and the entries in the  first row and the first column of $Z_2$ are zeros.
If the second row of $Z_2$ is not zero, we can choose
$\xi$ such that
$$
d_1 <
\left\{ |d_2\pm \xi \gamma_{22}|^2 + \xi^2 \sum_{j=3}^n |\gamma_{2j}|^2\right\}^{1/2}
\leq \|Z_1\pm \xi Z_2\|=d_1,
$$
which is a contradiction.
\end{proof}

\begin{proof}[Proof of Proposition \ref{prop:tea-para-not-inv} {\rm (a)}]
Assume that $T$ is nonzero and preserves TEA pairs.
If $T$ is not invertible, there is a nonzero $A \in \bM_{m,n}$
such that $T(A) = 0$. We may replace $T$ by
$X \mapsto  T(UXV)$ for $U \in \bU_m, V \in \bU_n$, and
assume that $A = \sum_{j=1}^m s_j(A) E_{jj}$.

{\bf Case 1.} Suppose that $s_1(A) > s_2(A)$.

We claim that $T(B) = 0$ for all $B$ of the form  $[0] \oplus B_1$.
Indeed, for $\xi\in \IR$ with sufficiently large magnitude,
$\xi A+B$ and $\xi A -  B$  attain the triangle equality
and so are $T(\xi A + B) = T(B)$ and $T(\xi A - B) = T(-B)$.
It follows that
$$0 = \|T(B) + T(-B)\| = \|T(B)\| + \|T(-B)\| = 2 \|T(B)\|.$$
Thus, $T(B) = 0$.
In particular, $T(E_{ij}) = 0$ for
$i\ge 2$ and  $j \ge 2$.
Repeating the above arguments with
 $T(E_{22})=0$ we conclude that all $T(E_{ij})=0$ except for possibly $(i,j)=(1,2)$ or $(2,1)$.

{\bf Case 1a.}
If $m \ge 3$, we arrive at the contradiction
 that $T(X) = 0$ for all $X$, since $T(E_{31})=0$ implies $T(E_{12})=0$ and
$T(E_{32})=0$ implies $T(E_{21})=0$.

{\bf Case 1b.}
If $n\geq 3$, then $T(E_{23})=0$ implies $T(E_{12})=0$, and $T(E_{13})=0$ implies $T(E_{21})=0$.

{\bf Case 1c.}  Suppose $\bM_{m,n} = \bM_2(\IC)$.
We have $T(E_{11}) = T(E_{22}) =0$.
Consider the  projection
$$
P=\frac{1}{2}\begin{pmatrix} 1 & 1 \\ 1 & 1\end{pmatrix}.
$$
Since $I_2-\frac{ 1+\sqrt{3}i }{2}P$ and $I_2 -\frac{ 1-\sqrt{3}i} {2}P$ attain the triangle equality,
so are $T(I_2-\frac{ 1+\sqrt{3}i }{2}P)= -\frac{ 1+\sqrt{3}i }{2}T(P)$ and
$T(I_2 -\frac{ 1-\sqrt{3}i} {2}P)=-\frac{ 1-\sqrt{3}i }{2}T(P)$, which implies $T(P) =0$. Hence
$T(E_{12}+E_{21})=0$.
Since $E_{12}+ \frac{ 1+\sqrt{3}i }{2}  E_{21}$ and $E_{12}+\frac{ 1-\sqrt{3}i }{2}E_{21}$
attain the triangle equality,
so do $T(E_{12}+ \frac{ 1+\sqrt{3}i }{2}  E_{21})$ and $T(E_{12}+\frac{ 1-\sqrt{3}i }{2} E_{21})$.
It follows that $T(E_{12})=T(E_{21})=0$.
Thus, $T$ is the zero map, which is a contradiction.

{\bf Case 2.} Suppose $s_1(A) = \cdots = s_k(A) > s_{k+1}(A)$ with
$k < m \le n$.
Then for $B = E_{mm}$, we have $\xi A + B$ and $\xi A - B$
attain the triangle equality for sufficiently large $\xi$, and
so are $T(\xi A + B) = T(B)$ and $T(\xi A - B) = -T(B)$.
Hence, $T(B) = T(E_{mm}) = 0$. By   Case 1, with $E_{mm}$ taking the role of $A$, we get a contradiction again.

{\bf Case 3.}
Suppose $s_1(A) = \cdots = s_m(A)$.  We can further
assume that $T(A) = 0$ with $A = \sum_{i=1}^m E_{ii}$.

{\bf Case 3a.} Suppose $n \geq m \geq 3$. Let $J, K$ be any   subsets of
$M=\{1,\ldots, m\}$ such that $J\cup K\neq M$.
Let $E_J =\sum_{j\in J} E_{jj}$  and  $E_K= \sum_{k\in K} E_{kk}$.
Now, $A-E_J$ and $A-E_K$ attain the triangle equality,
and  so do $T(A-E_{J}) = -T(E_{J})$ and $T(A-E_{K}) = -T(E_{K})$.
Consequently,
$$
\left\|\sum_{j\in J} T(E_{jj}) + \sum_{k\in K} T(E_{kk})\right\| =
\left\|\sum_{j\in J} T(E_{jj}) \right\| +
\left\|\sum_{k\in K} T(E_{kk})\right\|.
$$
An inductive argument with the fact $T(A)= \sum_{i=1}^{m} T(E_{ii})=0$ gives us that
$$
\|T(E_{11})\| = \|T(E_{22}) + \cdots + T(E_{mm})\|
= \|T(E_{22})\| + \cdots + \|T(E_{mm})\|.
$$
We have similar equalities with $E_{11}$ replaced by
$E_{jj}$ for $j=2,\ldots, m$.
Summing up these $m$ equalities, we have
$$
\sum_{j=1}^{m} \|T(E_{jj})\| = (m-1) \sum_{j=1}^{m} \|T(E_{jj})\|,
$$
which implies that
$T(E_{jj})=0$ for $j=1,\ldots, m$, and then Case 1 applies.

{\bf Case 3b.} Suppose $n > m =2$.
For any $\xi \in \IR$, choose $\gamma > (1+\xi^2)/2$.  Since
$\gamma A - E_{22} + \xi E_{23}$ and $\gamma A - E_{22} - \xi E_{23}$
attain the triangle equality,  so do their images
$- T(E_{22}) + \xi T(E_{23})$
and  $- T(E_{22}) - \xi T(E_{23})$.
 It follows from Lemma \ref{new},
$T(E_{23}) = 0$, and Case 1 applies.

{\bf Case 3c.}
Finally, suppose $\bM_{m,n}= \bM_2(\IC)$ and $T(A) = 0$ with $A = I_2$.
Now, $A-\frac{ 1+\sqrt{3}i }{2}E_{11}$
and $A -\frac{ 1-\sqrt{3}i} {2}E_{11}$
attain the triangle equality, and
so do $T(A- \frac{ 1+\sqrt{3}i }{2}E_{11})= -\frac{ 1+\sqrt{3}i }{2}T(E_{11})$ and
$T(A-\frac{ 1-\sqrt{3}i }{2}E_{11})=-\frac{ 1-\sqrt{3}i }{2}T(E_{11})$,
which implies $T(E_{11})=0$, and Case 1 applies.

Therefore, we will arrive at the conclusion that $T=0$ in any case when $T$ is not invertible.
\end{proof}

We need another lemma to prove Proposition \ref{prop:tea-para-not-inv}(b).

\begin{lemma}\label{lem:Sp-cases}
  Suppose $n\geq m\geq 2$.
   Let $T: \bM_{m,n}\to\bM_{m, n}$ be a  linear map sending parallel pairs to parallel pairs.
   Let $i,j, k, r, s, t$ below be any feasible indices.
\begin{enumerate}
  \item[{\rm (a)}] If $T(E_{ij}) = T(E_{kj})=0$ then $T(E_{ir})$ and $T(E_{kr})$ are linearly dependent.
   \item[{\rm (a')}] If $T(E_{ij}) = T(E_{ik})=0$ then $T(E_{rj})$ and $T(E_{rk})$ are linearly dependent.
  \item[{\rm (b)}] If $T((E_{ii} + E_{jj})\oplus F)  =0$ for some $F\in \bM_{m-2,n-2}$ with
  $\|F\|\leq 1$, then $T(E_{ii})$ and $T(E_{jj})$ are linearly dependent.
\end{enumerate}
Suppose, in addition,  $n\geq 3$.
\begin{enumerate}
  \item[{\rm (c)}] If $T(E_{ii} + E_{jj})=0$ then $T(E_{ii})$ and $T(E_{jk})$ are linearly dependent whenever $k\neq i$.
  \item[{\rm (d)}]
   If $T(E_{ij})=0$ then $T(E_{ir})$ and $T(E_{st})$ are linearly dependent whenever $r\neq j$, $s\neq i$ and $t\neq j, r$.
 \item[{\rm (d')}] If $T(E_{ij})=0$ then $T(E_{sj})$ and $T(E_{tr})$ are linearly dependent whenever $r\neq j$, $s\neq i$ and $t\neq i, s$.
\end{enumerate}
\end{lemma}
\begin{proof}
  (a)  It suffices to verify   that $T(E_{11})$ and $T(E_{21})$ are linearly dependent when $T(E_{12})=T(E_{22})=0$.
  By Example \ref{eg:parallel}, any two complex unitary matrices $U_1, U_2$ are parallel,
  and any two real orthogonal matrices $U_1, U_2$
 are parallel exactly when the spectrum $\sigma(U_1^{\TT} U_2)$ contains $1$ or $-1$.
It follows that for any scalars $\alpha, \beta, \gamma, \delta$, the  $m\times n$ matrices
$$
\frac{1}{\sqrt{|\alpha|^2 +|\beta|^2}}(\alpha E_{11} + \beta E_{21} - \overline{\beta} E_{12} + \overline{\alpha} E_{22})
$$
and
$$
\frac{1}{\sqrt{|\gamma|^2 +|\delta|^2}}(\gamma E_{11} + \delta E_{21} + \overline{\delta} E_{12} - \overline{\gamma} E_{22})
$$
are parallel,  and so are their scalar multiples.  It follows that
$T(\alpha E_{11} + \beta E_{21} - \overline{\beta} E_{12}+ \overline{\alpha} E_{22}) = \alpha T(E_{11}) + \beta T(E_{21})$
and
$T(\gamma E_{11} + \delta E_{21} + \overline{\delta} E_{12} - \overline{\gamma} E_{22}) = \gamma T(E_{11}) + \delta T(E_{21})$
are parallel.
By Lemma \ref{lem:par-1dim}, $T(E_{11})$ and $T(E_{21})$ are linearly dependent.

(b)  It suffices to verify   that $T(E_{11})$ and $T(E_{22})$ are linearly dependent,
when $T(A)=0$ for
 $A=  E_{11} + E_{22} + s_3 E_{33} + \cdots +  s_m E_{mm}$ with
  $1\geq  s_3 \geq s_4 \geq    \cdots \geq s_m \geq 0$.
For any scalars $\alpha, \beta, \gamma, \delta$, consider the $m\times n$ matrices
$B=\alpha E_{11} +\beta E_{22}$ and $C=\gamma E_{11} + \delta E_{22}$.
We claim that $T(B)$ and $T(C)$ are parallel.
If $B, C$ are parallel, then it is the case.
Otherwise, we can assume that $B=E_{11} + \mu E_{22}$ and $C=\nu E_{11} + E_{22}$ with $|\mu|, |\nu|<1$.
If $|1-\mu| \geq |1+\mu|s_3$, then $B - (1+\mu)A/2$ and $C$ are parallel, and so are $T(B)=T(B - (1+\mu)A/2)$ and $T(C)$.
If $|1-\nu| \geq |1+\nu|s_3$, then $B$ and $C - (1+\nu)A/2$ are parallel, and so are $T(B)$ and $T(C)=T(C - (1+\nu)A/2)$.
If $|1-\mu| < |1+\mu|s_3$ and $|1-\nu| < |1+\nu|s_3$ then $B - (1+\mu)A/2$ and
$C - (1+\nu)A/2$ are parallel, then so are $T(B)=T(B - (1+\mu)A/2)$
and $T(C)=T(C - (1+\nu)A/2)$.
 We   see that $T(B)$ and $T(C)$ are parallel in all cases.
By Lemma \ref{lem:par-1dim}, $T(E_{11})$ and $T(E_{22})$ are linearly dependent, as asserted.

(c) It suffices to verify $T(E_{11})$ and $T(E_{23})$ are linearly dependent when
   $T(A)=0$ where $A=E_{11}+ E_{22}$, due to (b).
  Consider $B=\alpha E_{11} + \beta E_{23}$ and $C=\gamma E_{11} + \delta E_{23}$ for any scalars $\alpha, \beta, \gamma$
  and $\delta$.
  If $B$ and $C$ are parallel, then so are $T(B)$ and $T(C)$.
  If $B$ is not parallel to $C$ then we can assume that $B= E_{11} + \beta E_{23}$ and $C= \gamma E_{11} + E_{23}$
  with $|\beta|, |\gamma|<1$.
  If $\gamma\neq 0$ then
  $B$ is parallel to
  $s\gamma A+ C$ for  $s \geq (1-|\gamma|^2)/2|\gamma|^2$.
  We see that $T(B)$ is parallel to $T(s\gamma A + C) = T(C)$.
In case when $\gamma=0$, we see that $T(B)$ is parallel to
  $T(C)+\epsilon T(E_{11})= T(\epsilon E_{11}+ E_{23})$ for all $\epsilon\neq 0$.
 Letting $\epsilon\to 0$, it follows from Lemma \ref{lem:limit} that  $T(B)$ and $T(C)$ are parallel.
  Therefore, in any case $T(B)$ is parallel to $T(C)$.
  By Lemma \ref{lem:par-1dim},  $T(E_{11})$ and $T(E_{23})$ are linearly dependent.

  (d) It suffices to verify that $T(E_{12})$ and $T(E_{21})$ are linearly dependent when $T(E_{23})=0$.
  To see this,  consider  any scalars $\alpha, \beta, \gamma$
  and $\delta$.  Let $\epsilon_k \to 0^+$ such that both $\beta + \epsilon_k$ and $\delta + \epsilon_k$ are nonzero
  for $k=1,2,\ldots$.  For
  each $k$, there is a big enough $t_k>0$ such that  $\alpha E_{12} +(\beta + \epsilon_k)(E_{21}+ t_k E_{23})$
and $\gamma E_{12} +(\delta + \epsilon_k)(E_{21}+ t_k E_{23})$ are parallel, and thus so are
$T(\alpha E_{12} +  (\beta + \epsilon_k)(E_{21}+ t_k E_{23})) = \alpha T(E_{12}) + (\beta + \epsilon_k)T(E_{21})$
and
$T(\gamma E_{12} +  (\delta + \epsilon_k)(E_{21}+ t_k E_{23})) = \gamma T(E_{12}) + (\delta + \epsilon_k)T(E_{21})$.
As $k\to \infty$,  it follows from  Lemma \ref{lem:limit}
  that $\alpha T(E_{12}) +  \beta  T(E_{21})$ and $\gamma T(E_{12}) +  \delta  T(E_{21})$
are parallel.  By Lemma \ref{lem:par-1dim},  $T(E_{12})$ and $T(E_{21})$ are linearly dependent.


By similar arguments, we can prove (a') and (d').
\end{proof}

\begin{proof}[Proof of Proposition \ref{prop:tea-para-not-inv} {\rm (b)}]
  Without loss of generality, we can assume that $n\geq m\geq 2$.
  Suppose that the nonzero linear map $T$ is not injective.  We need to show that the range space of $T$ has dimension one.
Let $A \in \bM_{m,n}$ have  norm one and $T(A) = 0$. We may replace $T$ by
$X \mapsto  T(UXV)$ for some suitable unitary matrices $U \in \bU_m, V \in \bU_n$, and
assume that $A = I_k\oplus A_1$ with $A_1\in \bM_{m-k, n-k}$ and $\|A_1\|< 1$.
Furthermore, we may assume that $k$ is the smallest positive integer such that $s_1(B) = \cdots = s_k(B)$
for every  nonzero matrix $B$ in the kernel of $T$.

{\bf Step 1.}   We claim that $A\neq I_2\oplus 0$.  Suppose otherwise $T(I_2\oplus 0)=0$, and thus
   $T(E_{11})=-T(E_{22})$ is nonzero.
By Lemma \ref{lem:Sp-cases}(c),
 $T(E_{11})$ and $T(E_{23})$ are linearly dependent, as well as $T(E_{22})$ and
  $T(E_{13})$ are linearly dependent.  Therefore, the linear span of $T(E_{11})= -T(E_{22}), T(E_{13}), T(E_{23})$
  has dimension  one.
By the minimality of $k=2$, $T(E_{13})\neq 0$, and thus $T(\xi E_{11} +  \eta E_{13})=0$ for some
nonzero $\xi, \eta\in\IF$.  Since the rank one matrix $B = \xi E_{11} + \eta E_{13}$
  satisfies  $s_1(B) > s_2(B)=0$, this conflicts with the minimality of $k$.

{\bf Step 2.}  We claim that $k=1$.  Suppose otherwise that
 $A=  E_{11} + E_{22} + a_3 E_{33} + \cdots +  a_m E_{mm}$ with
  $1\geq  a_3 \geq a_4 \geq    \cdots \geq a_m \geq 0$.
It follows from Lemma \ref{lem:Sp-cases}(b) that $T(E_{11})$ and $T(E_{22})$ are linearly dependent.
Consequently, there is a nontrivial linear combination $\xi E_{11} + \eta E_{22}$ belongs to the kernel of $T$.
The minimality assumption on $k$ ensures that $k=2$, and we can further assume $|\xi|=|\eta|=1$.
 Replacing $T$ by the map $X\mapsto T(UXV)$ for some
suitable $U\in \bU_m$ and $V\in \bV_n$,
we can assume that $T(E_{11} + E_{22})=0$. By  Step 1, we get a contradiction.

{\bf Step 3.}  It follows from Step 2 that
$$
A= [1] \oplus A_1\quad\text{with}\quad \|A_1\|< 1.
$$
Let
$B= [0]  \oplus B_1$ and $C=[0] \oplus C_1$
with $B_1, C_1\in \bM_{m-1,n-1}$.
For $\xi\in \IR$ with sufficiently large magnitude,
$\xi A+B$ and $\xi A +C$  are parallel,
and so are $T(\xi A + B) = T(B)$ and $T(\xi A +C) = T(C)$.
It follows from Lemma \ref{lem:par-1dim} that the space
$$
 \bV_1=
\{T(B): B = [0] \oplus B_1\in \bM_{m,n}\} = \mbox{Span}(K)
$$
for some $K\in \bM_{m,n}$.

{\bf Step 4.}
Let $[0]\oplus S_1\in\bM_{n, m}$ define the linear functional such that
$$
T(B) = \tr(S_1 B_1) K\quad\text{whenever}\ B = [0] \oplus B_1\in \bM_{m,n}.
$$
After replacing $T$ by the map  $X\mapsto T([1]\oplus U_1)X([1] \oplus V_1)$ for some unitary matrices $U_1\in \bU_{m-1}$
and $V_1\in \bU_{n-1}$, we can assume $[0]\oplus S_1 = \lambda_{22} E_{22} +  \cdots + \lambda_{mm} E_{mm}$
for some (maybe zero)
scalars $\lambda_{22}, \ldots, \lambda_{mm}$.   Thus,
$$
T(E_{ij})=0\quad\text{whenever  $i\neq j$ and $i,j\geq 2$.}
$$
We also know that
$$
T(E_{kk}) = \lambda_{kk} K, \quad\text{for $k=2, \ldots, m$}.
$$

\textbf{Step 5.}
By Step 4,
it remains to check those $T(E_{ij})$ with $i=1$ or $j=1$.
%
Since $T(E_{23})=0$, by Lemma \ref{lem:Sp-cases}(d, d') we
   have
\begin{align}
T(E_{1j}) \quad &\text{and}\quad T(E_{21})\quad\text{are linearly dependent when $j\neq 1,3$,} \label{eq:1j21}\\
T(E_{13}) \quad &\text{and}\quad T(E_{i1})\quad\text{are linearly dependent when $i\neq 1, 2$.} \label{eq:13i1}
\end{align}
We also know that
\begin{align}\label{eq:1ji1kk}
\dim \Span \{T(E_{1j}),  T(E_{i1}), T(E_{kk}): i\neq 2,  j\neq 3, k\neq 2, 3\}\leq 1,
\end{align}
 by an argument with $A$ replaced by $E_{23}$ as in
 Step 3.

\textbf{Case 5a.}  Suppose first that $T(E_{11}) = - T([0]\oplus A_1) = \lambda_{11} K\neq 0$, due to $T([0]\oplus A_1)\in \bV_1$.
It follows from \eqref{eq:1ji1kk} that  $ T(E_{ij})  = \lambda_{ij}K$
 for some scalar $\lambda_{ij}$, whenever $i\neq 2$ and $j\neq 3$.
 We need to check $T(E_{21})$ and $T(E_{13})$.

\indent\textbf{Subcase 5a.i.}
Suppose that $T(E_{22})= 0$.
Arguing as in Step 3 with $A$ replaced by $E_{22}$, we have
$T(E_{1j})=\lambda_{1j} K$ and $T(E_{i1})=\lambda_{i1} K$ with some scalars $\lambda_{1j}$ and $\lambda_{i1}$
for $j\neq 2$ and $i\neq 2$.  It remains to check $T(E_{21})$.

If $T(E_{12})=\lambda_{12}K\neq 0$, then by \eqref{eq:1j21}, $T(E_{21})=\lambda_{21}K$.
Suppose that
$T(E_{12})  =  0$.  It follows
from Lemma \ref{lem:Sp-cases}(a), $T(E_{21})$ linearly depends on $T(E_{11})$ and thus assumes
the form $\lambda_{21} K$ for some scalar $\lambda_{21}$.

\indent\textbf{Subcase 5a.ii.}
  Suppose that $T(E_{22})= \lambda_{22}K \neq 0$.

If $T(E_{12})=\lambda_{12} K\neq 0$, then $T(E_{21})=\lambda_{21} K$ by \eqref{eq:1j21}.
Replacing $T$ by a map $X\mapsto T(XV)$ with some suitable unitary $V=V_2\oplus I_{n-2}\in \bU_n$, we can assume that $T(E_{22})=0$.
In this new setting,
$T(E_{ij}) = \lambda'_{ij} K$ for $i,j=1,2$, and at least one of them is nonzero.
Note that both $T(E_{23})=0$ and $T(E_{13})$ assume  their original values.
By Lemma \ref{lem:Sp-cases}(a'), $T(E_{22})=T(E_{23})=0$ implies that $T(E_{12})$ and $T(E_{13})$ are linearly dependent.
It follows from Lemma \ref{lem:Sp-cases}(d) and $T(E_{22})=0$ that $T(E_{21})$ and $T(E_{13})$ are linearly dependent.
By Step 3 with $A$ replaced by $E_{22}$, we see that $T(E_{11})$ and $T(E_{13})$ are linearly dependent.
Consequently, $T(E_{13})=\lambda_{13}K$.

Now assume $T(E_{12})=0$.
Consider the map $T_\theta(X) = T(XV_\theta)$ with the unitary matrix
$V_\theta =\begin{pmatrix} \cos\theta & - \sin \theta \cr \sin \theta & \cos\theta\cr\end{pmatrix} \oplus I_{n-2}\in\bU_n$
for some $\theta\in (0,\pi/2)$ such that
\begin{align*}
T_\theta(E_{11}) &= \cos\theta T(E_{11}) - \sin\theta T(E_{12})= \cos\theta \lambda_{11}K \neq 0, \\
T_\theta(E_{12})  &= \sin \theta T(E_{11}) + \cos\theta T(E_{12}) = \sin \theta \lambda_{11}K\neq 0, \quad
T_\theta(E_{13}) = T(E_{13}),\\
T_\theta(E_{21}) &= \cos\theta T(E_{21}) - \sin\theta T(E_{22}),  \\
T_\theta(E_{22}) &= \sin\theta T(E_{21}) + \cos\theta T(E_{22})\neq 0, \quad
T_\theta(E_{23}) = T(E_{23})= 0.
\end{align*}
Note that   $T_\theta$ also preserves parallel pairs.  In particular, $T_\theta(E_{21})$ and
$T_\theta(E_{12})= \sin \theta \lambda_{11}K$ are linearly dependent by \eqref{eq:1j21}, and thus
 $T_\theta(E_{21})=\lambda'_{21} K$ for some scalar $\lambda'_{21}$.
 Consequently, $T(E_{21})= (T_\theta(E_{21}) + \sin\theta T(E_{22}))/\sin\theta = (\lambda'_{21} K+ \sin\theta \lambda_{22}K)/\sin\theta
 =\lambda_{21}K$ for some scalar $\lambda_{21}$.  Therefore,
 $T_\theta(E_{22})= \sin\theta T(E_{21}) + \cos\theta T(E_{22}) =\lambda'_{22}K\neq 0$ for some scalar $\lambda'_{22}$.
It then follows from the argument in the previous paragraph
 that $T_\theta(E_{13})= T(E_{13})=\lambda_{13}K$.

\medskip

\indent\textbf{Case 5b.}  Suppose that any $T(E_{jj}) = \lambda_{jj}K\neq 0$.
Replacing $T$ by a linear map $A\mapsto T(UAV)$ for some permutation matrices
$U\in \bU_m$ and $V\in \bU_n$, we can return to the Case 5a in which $T(E_{11})=\lambda_{11}K\neq 0$.

\indent\textbf{Case 5c.}  Suppose that all $T(E_{jj})=0$.  In this case, all $T(E_{1j})$ and $T(E_{i1})$ with all feasible $i,j$ are linearly
dependent, and thus the range space of the nonzero linear map $T$ has dimension  one.

The proof is complete.
  \end{proof}


\section{The case $(m,n)= (2,2)$}\label{S:2x2}
\setcounter{equation}{0}

In this section, we present the proofs of Theorems \ref{thm-parallel-C} and \ref{thm-parallel-R}.
First we present some notations and remarks.

Let $\bH_2$ (resp.\ $\bH_2^0$) be the real linear space of all $2\times 2$ Hermitian matrices (resp.\ with zero trace).
Let
$$
I_2 = \begin{pmatrix} 1 & 0 \cr 0 & 1\cr\end{pmatrix}, \
\mathcal{C}_1 = \mathcal{R}_1 = \begin{pmatrix} 1 & 0 \cr 0 & -1\cr\end{pmatrix}, \
\mathcal{C}_2 = \mathcal{R}_2 = \begin{pmatrix} 0 & 1 \cr 1 & 0\cr\end{pmatrix}, \
\mathcal{C}_3 = \begin{pmatrix} 0 & -i \cr i & 0 \cr\end{pmatrix}, \
\mathcal{R}_3 = \begin{pmatrix} 0 & 1 \cr -1 & 0 \cr\end{pmatrix}.
$$
Recall that $\bC = \{I_2/\sqrt 2, \mathcal{C}_1/\sqrt 2, \mathcal{C}_2/\sqrt 2, \mathcal{C}_3/\sqrt 2\}$
is an orthonormal basis of $\bM_2(\IC)$, while
$\bC_0 = \{ \mathcal{C}_1/\sqrt 2, \mathcal{C}_2/\sqrt 2, \mathcal{C}_3/\sqrt 2\}$ is an orthonormal basis of $\bH_2^0$,
 with respect to the inner product
$\langle A,B\rangle=\operatorname{trace}(B^*A)$.
On the other hand,
 $\bR=\{I_2/\sqrt 2, \mathcal{R}_1/\sqrt 2, \mathcal{R}_2/\sqrt 2, \mathcal{R}_3/\sqrt 2\}$
is an orthonormal basis of $\bM_2(\IR)$  with respect to the inner product $\langle A, B\rangle = \operatorname{trace}(B^{\TT} A)$.

\begin{remark}[Remarks on Theorem \ref{thm-parallel-C}]\label{rem:PC}
  \begin{enumerate}[{\rm (1)}]
\item  TEA preservers on $\bM_2(\IC)$ have the same form as those
of $\bM_{n}(\IC)$ for $n\ge 3$, whereas parallel preservers on $\bM_2(\IC)$ may have different forms.


\item Linear preservers of parallel pairs in (b.2) form a semigroup of operators.
         Each map $T$ in the semigroup is a composition of two maps of the form $X \mapsto UXV$ with $U, V \in \bU_2$,
             and a map of the form  $\tilde T$ described in \eqref{eq:(b.2)}.
The map $T$ is invertible if and only if  $\tilde T$ is.
Moreover,  every singular map in (b.2) is a limit of invertible maps in the semigroup.

\item There are maps of the form (b.1), which do not arise from (b.2).
For example, take $T(A) = (\tr A)E_{11}$.  Assume there are  unitary matrices $U_1, U_2, V_1, V_2\in \bU_2(\IC)$
such that the matrix representation of the linear map $\tilde{T}(A)= U_1T(U_2 AV_2)V_1$ assumes the form given in \eqref{eq:(b.2)}.
Note that the range of this map is spanned by the  matrix $uv^*$, where $u$ is the first column of $U_1$
and $v$ is the first column of $V_1^*$.  Since the range space is one dimensional, exactly one column of $\tilde{T}$ is nonzero.
But then the range space is spanned by exactly one of the four unitary matrices $I_2, \mathcal{C}_1, \mathcal{C}_2$ and $\mathcal{C}_3$, while the  matrix $uv^*$ has rank one.

\end{enumerate}
\end{remark}

\begin{remark}[Remarks on Theorem \ref{thm-parallel-R}]\label{rem:PR}
\begin{enumerate}[{\rm (1)}]

\item Both TEA and parallel preservers on $\bM_2(\IR)$
have more complicated structure compared with those on $\bM_n(\IR)$ for $n \ge 3$.
Of course, TEA preservers are parallel preservers. The only additional maps are those
 described in (b).

\item In (a.1), if $\{Z_1, Z_2\}$ is linearly dependent,
then $T$ has the form $A \mapsto \tr(FA)Z$ for some orthogonal matrix $F\in \bM_2$. However,
a general map of such form may not preserve TEA pairs unless $F=\alpha F_1 + \beta F_1(E_{12}-E_{21})$
for some  orthogonal matrix $F_1$ and scalars $\alpha, \beta$.

\item The maps in (a.2) form a semigroup of operators generated by invertible operators of the
form  $X \mapsto UXV$, and diagonal  operators $\tilde T$.
Each map in the semigroup is a composition of two maps of the form $X \mapsto UXV$,
             and a diagonal map $\tilde T$ arising from \eqref{eq:(a.2)}.
A map is invertible if and only if the map $\tilde T$ is. Moreover,
every singular map in (a.2) is a limit of invertible maps in the semigroup.

\item  There are maps carrying the form in (a.1),
which do not arise from (a.2).
For example, take $T(A) = (\tr A)E_{11}$ or $T(A) = (\tr A)E_{11} +(\tr (E_{12}-E_{21})A)E_{12}$.
\end{enumerate}
\end{remark}

We note that our proofs of the theorems
are computational, and utilize some techniques in matrix groups. We will first present some
auxiliary results in subsection \ref{Auxiliary results}, and then give the proofs of the theorems in the next two subsections.

\subsection{Auxiliary results}\label{Auxiliary results}

 In general, for a given parallel preserver $T$
 it may not be easy to find unitary matrices $U_1, U_2, V_1, V_2$
so that the modified map $\tilde T$ has the form \eqref{eq:(b.2)} in Theorem \ref{thm-parallel-C}(b.2).
By Proposition \ref{main3} below,  $T$ satisfies (b.2) if and only if there are
$U_3, U_4, V_3, V_4 \in \bU_2$ such that the map $\hat{T}$ defined by $A\mapsto U_3 T(U_4AV_4)V_3$
has a matrix representation with respect to the basis $\bC$ as
\begin{equation}\label{R-matrix}
\ \hskip .7 in [\hat{T}]_{\bC} =\begin{pmatrix}
d_0 & \begin{matrix}\alpha_1 i & \alpha_2 i & \alpha_3 i \end{matrix}\cr
\begin{matrix} 0 \\ 0 \\ 0\end{matrix} & \text{\Large $R$}   \cr\end{pmatrix}\in  \bM_4(\IC),
\end{equation}
for real numbers $d_0$, $\alpha_1$, $\alpha_2$, $\alpha_3$ and $R \in \bM_3(\IR)$.

Similarly, by Proposition \ref{main3},  $T$ satisfies (a.2) in Theorem \ref{thm-parallel-R} if and only if there are
$U_3, U_4, V_3, V_4 \in \bU_2$ such that the map $\hat{T}$ defined by $A\mapsto U_3 T(U_4AV_4)V_3$
has a matrix representation with respect to the basis $\bR$ as
\begin{equation}\label{S-matrix}
 [\hat{T}]_{\bR}=\begin{pmatrix}
d_0 & 0& 0  & \alpha_3\cr
0 & d_{11} & d_{12}& 0 \cr
0 & d_{21} & d_{22}& 0\cr
0 & 0 & 0 & d_{3} \cr
\end{pmatrix}\in \bM_4(\IR).
\end{equation}

\begin{proposition} \label{main3} Let $T: \bM_2(\IF) \rightarrow \bM_2(\IF)$ be a nonzero linear map.
\begin{itemize}
\item[(a)]  $T$ has the form in Theorem \ref{thm-parallel-C}{\rm (b.2)}   if and only if there are $U_3, U_4, V_3, V_4 \in \bU_2(\IC)$ such that the map $\hat{T}$ defined by $A\mapsto U_3 T(U_4AV_4)V_3$ has the form {\rm (\ref{R-matrix})}.

\item[(b)] $T$ has the form in Theorem \ref{thm-parallel-R}{\rm (a.2)}    if and only if there are $U_3, U_4, V_3, V_4 \in \bU_2(\IR)$ such that the map $\hat{T}$ defined by $A\mapsto U_3 T(U_4AV_4)V_3$ has the form {\rm (\ref{S-matrix})}.
\end{itemize}
\end{proposition}

\begin{proof}
We just need to verify the sufficiencies.

(a)
Let  $\IF = \IC$.  Using the linear isometric isomorphism which
\begin{quote}
sends
$A = a_1 \mathcal{C}_1/\sqrt 2 + a_2 \mathcal{C}_2/\sqrt 2 + a_3 \mathcal{C}_3/\sqrt 2 \in \bH_2^0$\quad to\quad
$v(A) = (a_1,a_2,a_3)^{\TT}\in \IR^3$,
\end{quote}
we get a one-one correspondence between
the automorphism $L_W: A \mapsto W^*AW$ of  $\bH_2^0$ for a unitary $W\in \bU_2(\IC)$ and an orthogonal matrix
$P_W\in \bU_3(\IR)$ with   determinant one such that
\begin{quote}
$L_W(A)=B$\quad if and only if \quad $P_W v(A) = v(B)$ for any $A\in\bH_2^0$.
\end{quote}

Note  that
 $L_W = L_{\mu W}$ for any $\mu \in \IC$ with $|\mu| = 1$.
We may assume that $W$ has a nonnegative $(1,1)$ entry and hence
$W = W_1 W_2 W_3$ with
 $W_1 = \begin{pmatrix} 1 & 0 \cr  0 & e^{i\phi_1}\cr\end{pmatrix}, W_2 =
\begin{pmatrix} \cos\phi_2 &  \sin \phi_2 \cr - \sin\phi_2 & \cos\phi_2 \cr
\end{pmatrix},$
and $W_3 = \begin{pmatrix} 1 & 0 \cr  0 & e^{i\phi_3}\cr\end{pmatrix}$.
(Here, if $W = (w_{ij})$, just choose $\phi_1, \phi_3$ such that
$e^{-i\phi_1} w_{21} \le 0$ and $e^{-i\phi_3} w_{12} \ge 0$.)
Then $L_W = L_{W_3} \circ L_{W_2} \circ L_{W_1}$.
The matrices corresponding to the transformations $L_{W_1}, L_{W_2}, L_{W_3}$ are the orthogonal matrices
$$P_{W_1} = [1] \oplus \begin{pmatrix}  \cos\phi_1 & \sin \phi_1 \cr
 -\sin \phi_1 & \cos \phi_1 \cr\end{pmatrix}, \
P_{W_2} = \begin{pmatrix} \cos 2\phi_2 & -\sin 2\phi_2 \cr
\sin 2\phi_2 & \cos 2\phi_2 \cr \end{pmatrix} \oplus [1], \
P_{W_3} = [1] \oplus \begin{pmatrix} \cos \phi_3 & \sin \phi_3 \cr
-\sin \phi_3 & \cos \phi_3 \cr\end{pmatrix},$$
respectively.
Thus, the map $L_W$ corresponds to the orthogonal matrix
$P_W = P_{W_3} P_{W_2} P_{W_1}$.

On the other hand,  any $P \in \bU_3(\IR)$ with positive determinant
can be written as the product of $P = P_{W_3} P_{W_2} P_{W_1}$ as above.
To see this, we may choose
suitable $\phi_3$ to construct
$P_{W_3}$ so that the first column of $P_{W_3}^{\TT}P$ has the form
$(\cos\theta, \sin \theta, 0)^{\TT}$. (Here, if $P = (p_{ij})$, we can choose
$(\cos \phi_3, -\sin \phi_3)$ to be a positive multiple of $(p_{21}, p_{31})$.)
Let $\phi_2 = \theta/2$ so that
$P_{W_2}^{\TT}P_{W_3}^{\TT} P = [1] \oplus Q$ with $Q \in \bU_2(\IR)$ with positive
determinant. Let $P_{W_1} = [1] \oplus Q$. Thus,
$P = P_{W_3} P_{W_2} P_{W_1} = P_W$, where $W = W_1W_2W_3$.

Suppose $\hat{T}$ is given by (\ref{R-matrix}).
By the singular value decomposition,
there are $Q_1,  Q_2 \in \bU_3(\IR)$ with $\det(Q_1) = \det(Q_2) = 1$
such that $R = Q_1DQ_2$ for a diagonal matrix $D = \diag(d_1, d_2, d_3)$
such that  $\det(R) = d_1d_2 d_3$.
Suppose $Q_1$ and $Q_2$ correspond to the maps
$X \mapsto P_1^*XP_1$ and $X \mapsto P_2^*XP_2$ with $P_1,P_2 \in \bU_2$.
Let $\Gamma = (\gamma_1, \gamma_2, \gamma_3)
= (\alpha_1, \alpha_2, \alpha_3)Q_2^{-1}$.
Then
$$[\hat{T}]_{\bC}=\begin{pmatrix} 1 & \cr & Q_1\end{pmatrix}
\begin{pmatrix} d_0 &   i\Gamma\cr
0 & D\cr\end{pmatrix}
\begin{pmatrix} 1 & \cr & Q_2\end{pmatrix},$$
which corresponds to the map $A \mapsto P_1^*(\tau(P_2^*AP_2))P_1$,
where $\tau$ is given by \eqref{eq:(b.2)}. Let
 $U_1=P_1U_3$, $U_2=U_4P_2$, $V_1=V_3P_1^*$ and $V_2=P_2^*V_4$. Then $\tilde T=\tau$ has the form
 \eqref{eq:(b.2)} in Theorem \ref{thm-parallel-C} .

(b)
Let $\IF = \IR$.
Suppose $\hat{T}$ has the form (\ref{S-matrix}).
By the singular value decomposition,  there are
$Q_j = \cos t_j I_2 + \sin t_j (E_{21}-E_{12}) \in \bU_2(\IR)$ for
$j = 1,2$ such that
$Q_2 \begin{pmatrix}
d_{11} &  d_{12} \cr d_{21} &  d_{22}\cr\end{pmatrix} Q_1 =
 \diag(\xi_1, \xi_2)$, where $\xi_1, \xi_2 \in \IR$ with $\xi_1 \xi_2 =
 d_{11}d_{22}-d_{12}d_{21}$.
Suppose $W_j = \cos (t_j/2) I_2 + \sin(t_j/2) (E_{12}-E_{21})$ for
$j = 1,2$.
Then, with respect to the basis $\bR$ of $\bM_2(\IR)$,  the matrix
representation of the
map $A \mapsto W_j A W_j^{\TT}$ has the form $[1] \oplus Q_j \oplus [1]$.
Hence, the map $L_1$ defined by $A \mapsto W_2 \hat T(W_1 A W_1^{\TT}) W_2^{\TT}$
 has the form
$$
[L_1]_{\bR}=
\begin{pmatrix}
d_1 & 0 & 0 & \alpha_3 \cr
0 & \xi_1 & 0 & 0 \cr
0 & 0   & \xi_2 & 0\cr
0 & 0 & 0 & d_3  \cr\end{pmatrix}.
$$
In other words,
$$
L_1(I_2) =  d_0 I_2, \quad
L_1(\mathcal{R}_1) = \xi_1\mathcal{R}_1, \quad  L_1(\mathcal{R}_2) = \xi_2\mathcal{R}_2, \quad L_1(\mathcal{R}_3) = \alpha_3 I_2 + d_3\mathcal{R}_3.
$$
Now, consider $L_2$ defined by $A \mapsto \mathcal{R}_2L_1(A\mathcal{R}_2)$.
Then
\begin{gather*}
L_2(I_2) = \xi_2 I_2, \quad
L_2(\mathcal{R}_1) = \mathcal{R}_2L_1(\mathcal{R}_3) = \alpha_3 \mathcal{R}_2 - d_2\mathcal{R}_1, \\
L_2(\mathcal{R}_2) = \mathcal{R}_2 L_1(I_2) =  d_0 \mathcal{R}_2, \quad
L_2(\mathcal{R}_3)= \mathcal{R}_2 L_1(\mathcal{R}_1) = -\xi_1 \mathcal{R}_3,
\end{gather*}
and thus
$$
[L_2]_{\bR}=\begin{pmatrix}
\xi_2 & 0        & 0   &0 \cr
0     & -d_2     & 0   & 0 \cr
0     & \alpha_3 & d_0 & 0 \cr
0     & 0        & 0   & -\xi_1\cr
\end{pmatrix}.
$$
 Let $Q_j = \begin{pmatrix}\cos t_j & - \sin t_j \\ \sin t_j & \cos t_j\end{pmatrix}$
for $j = 3,4$ such that
$Q_4 \begin{pmatrix} -d_2     & 0  \cr \alpha_3 & d_0\cr\end{pmatrix} Q_3
= \begin{pmatrix}\xi_3&\\ & \xi_4\end{pmatrix}$ with $\xi_3\xi_4 = -d_0d_2$.
Let
$$
 W_j = \begin{pmatrix} \cos (t_j/2) & - \sin (t_j/2) \\ \sin (t_j/2) & \cos (t_j/2)\end{pmatrix}, \quad j  = 3,4.
$$
Then the map
$L_3(A) = W_4 L_2(W_3AW_3^{\TT}) W_4^{\TT}$ satisfies
$$
L_3(I_2) = \xi_2 I_2,  \quad L_3(\mathcal{R}_1) = \xi_3 \mathcal{R}_1,
\quad L_3(\mathcal{R}_2) = \xi_4 \mathcal{R}_2, \quad L_3(\mathcal{R}_3) = -\xi_1 \mathcal{R}_3,
$$
and thus
$$
[L_3]_{\bR}=\begin{pmatrix}
\xi_2 & 0        & 0   &0 \cr
0     & \xi_3     & 0   & 0 \cr
0     & 0 & \xi_4 & 0 \cr
0     & 0        & 0   & -\xi_1\cr
\end{pmatrix}.
$$
Observe that
\begin{align*}
L_3(A) &=
W_4 L_2(W_3AW_3^{\TT}) W_4^{\TT}
= W_4 \mathcal{R}_2 L_1(W_3 AW_3^{\TT} \mathcal{R}_2 )W_4^{\TT} = W_4 \mathcal{R}_2 W_2 \hat T (W_1 W_3 A W_3^{\TT} \mathcal{R}_2 W_1^{\TT})W_2^{\TT} W_4^{\TT} \\
&=  W_4 \mathcal{R}_2 W_2 U_3T (U_4W_1 W_3 A W_3^{\TT} \mathcal{R}_2 W_1^{\TT}V_4)V_3 W_2^{\TT} W_4^{\TT}.
\end{align*}
With $U_1 = W_4 \mathcal{R}_2 W_2 U_3$, $U_2= U_4W_1 W_3$, $V_1 = V_3 W_2^{\TT} W_4^{\TT}$ and $V_2= W_3^{\TT} \mathcal{R}_2 W_1^{\TT}V_4$,
the map $\tilde T =L_3 $ sending $A$ to $U_1T(U_2AV_2)V_1$ has the form \eqref{eq:(a.2)} in Theorem \ref{thm-parallel-R}.
\end{proof}

It is worthwhile to mention that the proof of the complex case in Proposition \ref{main3} uses the fact
that the quotient group $\bU_2(\IC)/\bU_1(\IC)$ is isomorphic to the group of matrices
in $\bU_3(\IR)$ with positive determinant. However, such an isomorphism does not exist in the real case.

In the following, we show that the maps given by   Theorem \ref{thm-parallel-C}(b.2) and   Theorem \ref{thm-parallel-R}(a.2)
preserve parallel pairs.  Note that a map of the form $A \mapsto \gamma UAV$ will preserve parallel pairs, and a
 composition of parallel pair
preserves is a parallel pair preserver. Therefore, we may assume $T=\tilde T$ in Theorem \ref{thm-parallel-C}(b.2) and in Theorem \ref{thm-parallel-R}(a.2).

\begin{lemma} \label{scheme}
Suppose $G_1 = U(s_1 E_{11} + s_2 E_{22})V$
and $G_2 = \mu U(s_2 E_{11}+ s_1 E_{22})V$ for some  $U, V \in \bU_2$,
$s_1 > s_2 \ge 0$, and $\mu \in \mathbb{F}$ with $|\mu| = 1$.
Then
\begin{align}\label{eq:para-cond}
a_1G_1 + a_2 G_2
\ \text{and}\ b_1 G_1 + b_2 G_2\ \text{are parallel} \quad \Longleftrightarrow \quad
 (|a_1| - |a_2|)(|b_1|- |b_2|)\geq 0,
\end{align}
  for any $a_1, a_2, b_1, b_2 \in\mathbb{F}$.
In particular, the equivalence \eqref{eq:para-cond} holds when one of the
following condition is valid.
\begin{itemize}
\item[{\rm (a)}] $G_1G_2^* = \zeta I$ with $\zeta \ne 0$ and $|\det(G_1)|=|\det(G_2)|$.

\item[{\rm (b)}] $G_1 G_2^* = G_1^*G_2 = 0$ and
$\tr(G_1G_1^*) = \tr(G_2G_2^*)>0$.
\end{itemize}
\end{lemma}
\begin{proof}
Observe
$$a_1G_1 + a_2 G_2 = U((a_1s_1+\mu a_2s_2)E_{11} + (a_1s_2+\mu a_2s_1)
E_{22})V = U(\mu_1 E_{11} + \mu_2 E_{22})V$$
where $\mu_1 =  a_1s_1+\mu a_2s_2$ and
$\mu_2 =a_1s_2+\mu a_2s_1$
with
$$|\mu_1|^2-|\mu_2|^2 = (s_1^2-s_2^2)(|a_1|^2-|a_2|^2).$$
Similarly,
$$
b_1 G_1 + b_2 G_2 = U(\nu_1 E_{11} + \nu_2 E_{22})V
$$
where $\nu_1 = b_1s_1+\mu b_2s_2$ and
$\nu_2 = b_1s_2+\mu b_2s_1$
with
$$|\nu_1|^2-|\nu_2|^2 = (s_1^2-s_2^2)(|b_1|^2-|b_2|^2).$$
If $|a_1|\geq |a_2|$ and $|b_1|\geq |b_2|$, then $|\mu_1|\ge|\mu_2|$ and $|\nu_1|\ge |\nu_2|$;
if $|a_1|\leq |a_2|$ and $|b_1|\leq |b_2|$, then $|\mu_1|\le|\mu_2|$ and $|\nu_1|\le |\nu_2|$.
In both cases, $a_1 G_1 + a_2 G_2$ and $b_1 G_1 + b_2 G_2$ are parallel.  In the other  cases, we see that
$a_1 G_1 + a_2 G_2$ and $b_1 G_1 + b_2 G_2$ are not parallel to each other.
\end{proof}

\begin{lemma}\label{lem:2x2-suff} The map $T$ in Theorem \ref{thm-parallel-C} {\rm (b.2)}
(or  Theorem \ref{thm-parallel-R} {\rm (a.2)}) preserves parallel pairs.
\end{lemma}
\begin{proof}
Note that $T$ lies in the closure of the invertible linear maps. We may assume that $T$ is invertible by Lemma \ref{lem:limit}.

Suppose $\IF=\IC$ first, and $T: \bM_2(\IC)\to \bM_2(\IC)$ is defined by
 \begin{align*}
  T(I_2) = I_2, \quad T(\mathcal{C}_j) = i\gamma_j I_2 + d_j \mathcal{C}_j\quad\text{for $j=1,2,3$,}
\end{align*}
where  $d_1, d_2, d_3, \gamma_1,\gamma_2,\gamma_3 \in \mathbb{R}$. 
Let
$$
\mu_1 =d_1 + i\gamma_1, \quad \mu_2 = i\gamma_2 - \gamma_3, \quad \mu_3 = d_2 + d_3, \quad
\mu_4 = d_2  - d_3.
$$
Then,
\begin{gather}
2T(E_{11})  = \begin{pmatrix} 1 +  \mu_1 & 0 \cr 0 & 1-\bar \mu_1\cr \end{pmatrix},  \quad
2T(E_{22})  = \begin{pmatrix} 1 - \mu_1 & 0 \cr 0 & 1+\bar \mu_1 \cr \end{pmatrix},\notag \\
2T(E_{12})  = \begin{pmatrix} \mu_2 &  \mu_3  \cr \mu_4 & \mu_2 \cr \end{pmatrix}, \qquad \ \qquad
2T(E_{21})  = \begin{pmatrix} -\bar\mu_2  &  \mu_4 \cr \mu_3 & -\bar \mu_2 \end{pmatrix}. \label{eq:TE-ij}
\end{gather}

Suppose $X$ and $Y$ in $\bM_2$ are parallel.
We will show that $T(X)$ and $T(Y)$ are parallel.
Let $X=U\diag(a_1, a_2)V^*$ with $a_1\geq a_2\geq 0$  and
 $Y=U\diag(b_1,b_2)V^*$ with $|b_1|\geq |b_2|$ for some $U,V\in\bU_2(\IC)$.
 In this way,
$$
T(X)= a_1 T(UE_{11}V^*) + a_2 T(UE_{22}V^*) \quad\hbox{ and }
\quad T(Y)= b_1 T(UE_{11}V^*) + b_2 T(UE_{22}V^*).
$$
We may write
$U=\begin{pmatrix} \alpha & -\beta e^{i\theta_1} \\ \overline{\beta} & \overline{\alpha}e^{i\theta_1}\end{pmatrix}$ and
$V=\begin{pmatrix} \gamma & -\delta e^{i\theta_2} \\ \overline{\delta} & \overline{\gamma}e^{i\theta_2}\end{pmatrix}$ with
some $\theta_1,\theta_2\in\IR$ and $|\alpha|^2 + |\beta|^2 =
|\gamma|^2 + |\delta|^2 =1$.
Observe that
\begin{align*}
&\ 2 T(UE_{11}V^*) = 2T(\alpha \overline{\gamma} E_{11} + \overline{\beta}\delta E_{22} + \overline{\beta}\overline{\gamma} E_{21} +
\alpha\delta E_{12})\\
 = &\ \alpha \overline{\gamma}\begin{pmatrix} 1+\mu_1 & 0 \cr 0 & 1-\bar\mu_1\cr\end{pmatrix}
+ \overline{\beta}\delta \begin{pmatrix} 1-\mu_1 & 0 \cr 0 & 1+\bar \mu_1\cr\end{pmatrix}
+  \overline{\beta}\overline{\gamma} \begin{pmatrix} -\bar \mu_2 & \mu_4 \cr
\mu_3 & -\bar \mu_2 \cr\end{pmatrix}
+ \alpha\delta  \begin{pmatrix} \mu_2 & \mu_3 \cr
 \mu_4 & \mu_2 \cr\end{pmatrix},
 \intertext{and}
&\ 2 T(UE_{22}V^*) = 2T(\beta\overline{\delta}e^{i(\theta_1-\theta_2)} E_{11} + \overline{\alpha}\gamma e^{i(\theta_1-\theta_2)} E_{22}
- \overline{\alpha}\overline{\delta} e^{i(\theta_1-\theta_2)} E_{21} - \beta\gamma e^{i(\theta_1-\theta_2)} E_{12})\\
= &\ e^{i(\theta_1-\theta_2)}  \left[ \beta\overline{\delta} \begin{pmatrix} 1+\mu_1 & 0 \cr 0 & 1-\bar\mu_1\cr\end{pmatrix}
+ \overline{\alpha}\gamma \begin{pmatrix} 1-\mu_1 & 0 \cr 0 & 1+\bar \mu_1\cr\end{pmatrix}
- \overline{\alpha}\overline{\delta}  \begin{pmatrix} -\bar \mu_2  &  \mu_4 \cr
\mu_3 & -\bar \mu_2 \cr\end{pmatrix}
- \beta\gamma   \begin{pmatrix} \mu_2 & \mu_3 \cr
 \mu_4& \mu_2 \cr\end{pmatrix}\right] ,
\end{align*}
 respectively. In particular, if
$$d_{11} = \alpha \overline{\gamma}(1+\mu_1) + \overline{\beta}\delta (1-\mu_1)
- \overline{\beta}\overline{\gamma} \bar \mu_2+ \alpha\delta  \mu_2, \qquad
d_{12} = \overline{\beta}\overline{\gamma}\mu_4 + \alpha\delta \mu_3,$$
$$d_{21} =  \overline{\beta}\overline{\gamma} \mu_3 +\alpha\delta \mu_4, \quad
\hbox{ and } \quad
d_{22} = \alpha \overline{\gamma}(1-\bar \mu_1)
 + \overline{\beta}\delta(1+\bar \mu_1) - \overline{\beta}\overline{\gamma}\bar\mu_2 + \alpha\delta\mu_2,$$
then
\begin{gather*}
G_1 =2 T(UE_{11}V^*)
= \begin{pmatrix} d_{11} & d_{12} \cr d_{21} & d_{22}\end{pmatrix},
\qquad
G_2^* = 2 T(UE_{22}V^*)^*
= e^{i(\theta_2-\theta_1)}  \begin{pmatrix} d_{22} & -d_{12} \cr - d_{21} & d_{11}\cr
 \end{pmatrix}.
\intertext{Thus,}
  |\det(G_1)|=|\det(G_2)|, \quad\text{and}\quad G_1G_2^* = \zeta I\ \text{with
 $\zeta = e^{i(\theta_2-\theta_1)} (d_{11} d_{22} - d_{12} d_{21})$.}
 \end{gather*}
  If $\zeta \ne 0$, then
 $T(X), T(Y)$ are parallel by Lemma \ref{scheme} (a).
Otherwise,  we can check that $G_1^*G_2 = 0$ and $\tr(G_1G_1^*)  = \tr(G_2G_2^*)>0$ (since $T$ is invertible),
and Lemma \ref{scheme} (b) applies.

Suppose now that $\IF=\IR$, and $T: \bM_2(\IR)\to \bM_2(\IR)$ is defined by
  \begin{align*}
  T(I_2) &= I_2, \quad T(\mathcal{R}_1) =  d_1 \mathcal{R}_1, \quad T(\mathcal{R}_2) =  d_2 \mathcal{R}_2 \quad\text{and}\quad T(\mathcal{R}_3) =  d_3 \mathcal{R}_3,
\end{align*}
where $d_1,d_2,d_3$ are some real numbers.
Extend $T$ to a complex linear map $\tau: \bM_2(\IC)\to \bM_2(\IC)$ by
linearity.  Then $\tau(\mathcal{C}_3) = \tau(i\mathcal{R}_3)
= i\tau(\mathcal{R}_3) = id_3\mathcal{R}_3 = d_3\mathcal{C}_3$.
It follows from the complex case that $\tau$ sends
parallel pairs to parallel pairs, and so does its restriction $T$.
\end{proof}


Recall that $\mathcal{P}(A) = \{B\in \bM_2(\IF): A \text{ is parallel to }  B\}.$ We have the following.

\begin{lemma} \label{2X2}
Let $A \in \bM_{2}(\IF)$ and  $0\leq a<1$.
\begin{enumerate}
\item[{\rm (a)}]
$\mathcal{P}(I_2)$ consists of  complex normal matrices if  $\mathbb{F}=\mathbb{C}$, and    real symmetric matrices if $\mathbb{F}=\mathbb{R}$.

\item[{\rm (b)}]
Suppose $A\in \bM_2(\IC)$ is a normal matrix with equal singular values
\begin{align}\label{eq:s1=s2}
s_1(r I_2+iA)=s_2(r I_2+iA) \quad\text{for all $r\in\mathbb{R}$.}
\end{align}
 Then  $A =\mu I_2$ for some $\mu \in\mathbb{C}$, or $A=ipI_2+H$ for some
 $p\in\mathbb{R}$ and a Hermitian matrix $H\in \bH_2^0$  with zero trace.
\item[{\rm (c)}]
$\mathcal{P}(\diag(1, a)) =  \{\diag(x_1, x_2): x_1, x_2\in\mathbb{F}, |x_2| \leq |x_1|  \}$.
\item[{\rm (d)}]
If $\diag(1, a)+tA\in \mathcal{P}(\diag(1, a))$ for all $t\in \mathbb{F}$,
then $A=\diag(y_1, ay_1)$ with $y_1 \in\mathbb{F}$.
\end{enumerate}
\end{lemma}

\begin{proof}
(a) Suppose $X\in \mathcal{P}(I_2)$. Then there are $P, Q \in \bU_2$ such that  $PXQ$ and $PI_2Q$ are in diagonal forms, and so is $(PI_2Q)^*(PXQ) = Q^*XQ$. Thus, $X$ is complex normal if  $\mathbb{F}=\mathbb{C}$, and $X$ is real  symmetric if $\mathbb{F}=\mathbb{R}$.
The other inclusions are obvious.

(b) By (a), $A = U\diag(\mu, \nu) U^*$ for some $U \in \bU_2(\IC)$.
Since $rI_2+iA$ have identical singular values for any $r \in \IR$,
we have $|r+i\mu| = |r+ i\nu|$  for all $r\in \IR$.
It follows $\operatorname{Im}\, (\mu) =\operatorname{Im}\, (\nu)=p$  and   $\operatorname{Re}\, (\mu) =\pm\operatorname{Re}\ (\nu)=\pm q$.
Hence, either $A=\mu I$ or $A = U\begin{pmatrix} q + pi & 0\\ 0 & -q +pi\end{pmatrix}U^* = ipI_2 + H$ with
$H=U\begin{pmatrix} q  & 0\\ 0 & -q\end{pmatrix}U^*\in \bH_2^0$.

(c)
Suppose $X\in \mathcal{P}(\diag(1, a))$, $0\leq a<1$. Then there are $U, V\in \bU_2$ such that
$$
\begin{pmatrix} 1 & 0 \cr 0 & a \cr\end{pmatrix}=U\begin{pmatrix} 1 & 0 \cr 0 & a_1 \cr\end{pmatrix}V \quad\text{and}\quad
X=U\begin{pmatrix} x_1 & 0 \cr 0 & x_2 \cr\end{pmatrix}V,
$$
where $|a_1|=a$ and $|x_2|\leq |x_1|$. Consider the first row of $U^*\diag(1, a)=\diag(1, a_1)V$. We derive $U$ and  $V$ are in diagonal forms. Thus, $X=\diag(y_1, y_2)$ with $|y_2|\leq |y_1|$.

(d) Let $t=|t|e^{i\theta}$. By assumption and (c),  $A=\diag(y_1, y_2)$ with $|1+|t| e^{i\theta} y_1|  \ge |a+|t|e^{i\theta} y_2|$, and thus
$$
0 \le |1+|t| e^{i\theta} y_1|^2 - |a+|t|e^{i\theta} y_2|^2
 = (1-a^2) +    2|t|\operatorname{Re}\, (e^{i\theta} (y_1 - ay_2)) + |t|^2(|y_1|^2 - |y_2|^2)
 $$
for any  $t\in \mathbb{F}$.
Choosing  real $\theta_\pm$ such that $e^{i\theta_\pm}(y_1-ay_2) = \pm |y_1-ay_2|$, we have
$$
0 \le (1-a^2)\pm    2|t||y_1 - ay_2| + |t|^2(|y_1|^2 - |y_2|^2) \quad\text{for all $|t|\geq 0$.}
 $$
It follows
$$0 \ge   |y_1-ay_2|^2 - (1-a^2)(|y_1|^2 - |y_2|^2) = |ay_1-y_2|^2.$$
Thus, $y_2=ay_1$ and $A = y_1\diag(1,a)$.
\end{proof}

\subsection{Proof of Theorem \ref{thm-parallel-C}}

\begin{proof}[Proof of Theorem \ref{thm-parallel-C}{\rm (b)}]
For the sufficiency, $T$ clearly preserves parallel pairs if $T$ has the form in (b.1).  If $T$ has the form in (b.2), by Lemma \ref{lem:2x2-suff}, $T$ also preserves parallel pairs.  

For the necessity, suppose $T$  preserves parallel pairs.
If $T$ sends every unitary matrix to zero, then $T$ is a zero map.
We assume below $T(P) \ne 0$ for some $P \in \bU_2$.
We may replace the map by $X \mapsto UT(PX)V/\gamma$ and assume that $T(I) = \diag(1,a)$ with $0\leq a \leq 1$.

{\bf Case 1.} Suppose $0\leq a < 1$.
Let $X = a_1\mathcal{C}_1 + a_2 \mathcal{C}_2 + a_3 \mathcal{C}_3$ with  $T(X) = Y \ne 0$.
Note that $X$ is hermitian.
For any $t\in  \IC$, the normal matrix $I + t X$ and $I$ are parallel,
and hence $T(I+t X) =  \diag(1,a) + t Y$ and $T(I) =  \diag(1,a)$ are parallel.
Thus, by Lemma \ref{2X2}(d),  $Y = y_1\diag(1,a) = y_1 T(I)$.
Hence, the range space of $T$ is spanned by $T(I)$, and thus $T$ has the form (b.1).

{\bf Case 2.} We may  assume that $T(U)$ has identical singular values;
in other words, $T(U)$ is a multiple of a unitary matrix
for every unitary matrix $U\in \bU_2(\IC)$.  Otherwise, we will go back to Case 1.
Suppose further that   $T(I) = I$. Since $\mathcal{C}_j$ and $I$ are parallel, $T(\mathcal{C}_j)$ is
 parallel to $I$, and thus a complex normal matrix by Lemma \ref{2X2}(a).
Consequently,  $T(rI_2+i\mathcal{C}_j)=rI_2+iT(\mathcal{C}_j)$ is a complex normal matrix, while it is also
 a multiple of a unitary matrix inherited from $rI_2+i\mathcal{C}_j$ for $j = 1,2,3$ and $r \in \IR$.
By Lemma \ref{2X2}(b), $T(\mathcal{C}_j) =\alpha_j I_2=\alpha_j T(I_2)$ or $T(\mathcal{C}_j) = i\gamma_j I_2+ H_j$,
 where $H_j$ is a Hermitian matrix with zero trace.
We are going to verify that either $T(\mathcal{C}_j)=\alpha_j I_2$ for $j=1,2,3$, or $T(\mathcal{C}_j)=i\gamma_j I_2+ H_j$ for $j=1,2,3$.

Suppose $T(\mathcal{C}_1)= \alpha_1 I_2$ with $\alpha_1 = a + ib$ and $a\neq 0$.
Assume that  $T(\mathcal{C}_2)= i\gamma_2 I_2 + H_2$
for some Hermitian $H_2$ with eigenvalues $\pm h_2$.  Since  $r\mathcal{C}_1 + \mathcal{C}_2$ is a multiple of a unitary matrix,
so is  $T(r\mathcal{C}_1 + \mathcal{C}_2) = r(a+ib)I_2 + ir_2I_2 + H_2$ for all $r\in\IR$.  It follows
$|(ra + h_2)+ i(rb+\gamma_2)| = |(ra - h_2)+ i(rb+\gamma_2)|$ for all $r\in \mathbb{R}$.
Hence $h_2=0$, and $T(\mathcal{C}_2)=i\gamma_2 I_2$.
Similarly, assume that  $T(\mathcal{C}_3)= i\gamma_3 I_2 + H_3$
for some Hermitian $H_3$ with eigenvalues $\pm h_3$.
Since $r\cC_1 + \cC_3$
is a multiple of a unitary matrix, so
is $rT(\cC_1) + T(\cC_3)= r(a+ib)I_2 + ir_3I_2 + H_3$ for all $r\in\IR$.  It follows
$$
|(ra+h_3) + i(rb+\gamma_3)| = |(ra-h_3)+i(rb+\gamma_3)|\quad\text{foll all $r\in \IR$.}
$$
It forces $h_3=0$, and thus $T(\cC_3)=i\gamma_3 I_2$.
Therefore,  $T(\mathcal{C}_j)=\alpha_j I_2$ for $j=1,2,3$ in this case.

The situations when $T(\mathcal{C}_j)=\alpha_j I_2$ with $\alpha_j =c_j +id_j$ and $c\neq 0$ for $j=2,3$ are similar.
Therefore,   $T$ has the form given in (b.1).  Otherwise, the matrix of $T$ with respect to
the basis $\{I/\sqrt 2, \mathcal{C}_1/\sqrt 2, \mathcal{C}_2/\sqrt 2, \mathcal{C}_3/\sqrt 2\}$ has the form
$$
\begin{pmatrix} 1 & i\gamma_1 \  i\gamma_2 \ i\gamma_3 \cr
0 & \cr
0 & R \cr
0 & \cr\end{pmatrix}
$$
for some real scalars $\gamma_1,\gamma_2, \gamma_3$ and a real matrix $R$.
Proposition \ref{main3}(a) finishes the proof.
\end{proof}

\begin{proof}[Proof of Theorem \ref{thm-parallel-C}{\rm (a)}]
For the sufficiency, if $T$ has the form
(i) $A \mapsto UAV$ or
(ii) $A \mapsto UA^{\TT}V$, then $T$ clearly preserves TEA pairs.

For the necessity, by Proposition \ref{prop:tea-para-not-inv}(a),
a nonzero complex linear map preserves TEA pairs will be invertible.
Clearly, it also preserves parallel pairs. So, it
has the form described in Theorem \ref{thm-parallel-C} (b.2) with an invertible matrix in \eqref{eq:(b.2)}.
We will show that if such map $\tilde T$ is
a unital map (i.e., $d_0=1$) satisfying the condition that
\begin{align}\label{eq:tau1}
\|\tilde T(A)+ \tilde T(B)\| = \|\tilde T(A)\| + \|\tilde T(B)\| \quad\text{whenever}\quad \|A + B\| = \|A\| + \|B\|,
\end{align}
then it is either
the  map $\tilde T(A)=WAW^*$  or the transpose map $\tilde T(A)= WA^{\TT}W^*$ for a $W\in\bU_2(\IC)$.
Consequently, $T$ will reduce to the standard form of
$A \mapsto \gamma UAV$ or $A \mapsto \gamma UA^{\TT}V$ for some $\gamma>0$.

Observe that
 $\|I+\mathcal{C}_j\| = \|I\| + \|\mathcal{C}_j\|$ implies $\tilde T(I)=I$ and $\tilde T(\mathcal{C}_j)=i\gamma_j I + d_j\mathcal{C}_j$ form a TEA pair, namely,
$$
\|I + i\gamma_j I + d_j \mathcal{C}_j\| = \|I\| + \|i\gamma_j I + d_j \mathcal{C}_j\|\quad \text{for $j=1,2,3$.}
$$
Direct verification of the corresponding  equalities show
that $\gamma_j=0$  for all $j=1,2,3$.
Moreover, $A = 2E_{11}$ and $B = 2(E_{11}+i E_{22})$ satisfy
$\|A+B\| = \|A\|+\|B\|$, and by \eqref{eq:tau1} so are the images
$\begin{pmatrix} 1+d_1 & 0\\ 0 & 1-d_1\end{pmatrix}$ and
 $\begin{pmatrix} 1+d_1 +i(1-d_1) & 0 \\ 0 & 1-d_1 + i(1+d_1)\end{pmatrix}$  via \eqref{eq:TE-ij}.
It follows $d_1 = \pm1$.

Define $$\tau_1(A)=U\tilde T(UAU^*)U^*  \quad\mbox{and}
\quad  \tau_2(A)=V\tilde T(VAV^*)V^*,$$ where
$U=\frac{1}{\sqrt{2}}\begin{pmatrix} 1 & 1\\ 1 & -1\end{pmatrix}$
and $V=\frac{1}{\sqrt{2}}\begin{pmatrix} 1 & i\\ i & 1\end{pmatrix}$, respectively.
Then $\tau_i$ also satisfies \eqref{eq:tau1} as $\tilde{T}$. Now,
\begin{gather*}
\tau_1(I) = I, \quad \tau_1(\mathcal{C}_1) = d_2 \mathcal{C}_1,  \quad \tau_1(\mathcal{C}_2) = d_1 \mathcal{C}_2,  \quad \tau_1(\mathcal{C}_3) = d_3\mathcal{C}_3 \\
\tau_2(I) = I, \quad \tau_2(\mathcal{C}_1) = -d_3 \mathcal{C}_1,  \quad \tau_2(\mathcal{C}_2) = d_2 \mathcal{C}_2,  \quad \tau_2(\mathcal{C}_3) = -d_1\mathcal{C}_3.
\end{gather*}
Applying the arguments on $\tilde T$ to $\tau_i$,
,  we see that both $d_2, d_3$ assume the values $\pm1$.
In fact, we may assume $d_1=1$; otherwise, we can replace $\tilde T$ with the map
$A\mapsto U_1\tilde T(A)U_1^*$,
where $U_1=\begin{pmatrix} 0 & 1\\ 1 & 0\end{pmatrix}$.
Similarly, we may   assume $d_2=1$;
for else, we can replace $\tilde T$ with the map $A\mapsto U_2\tilde T(A)U_2^*$,
where $U_2=\begin{pmatrix} 1 & 0\\ 0 & -1\end{pmatrix}$.
Finally, we may also assume $d_3 =1$, because we can compose $\tilde T$ with the transpose map if necessary.
Hence, $\tilde T$ is the identity map, and the necessity of Theorem \ref{thm-parallel-C}(a) is verified.
\end{proof}

\subsection{Proof of Theorem \ref{thm-parallel-R}}

\begin{proof}[Proof of Theorem \ref{thm-parallel-R} {\rm (b)}]
For the sufficiency, if $T(A) = (\tr AF)Z$  for some nonzero $F, Z \in \bM_2(\IR)$, then $T$ clearly preserves parallel pairs. Suppose $T$ has
the form in (a.1); namely, $T(A) = \tr (F_1A)Z_1 + \tr(F_2A)Z_2$.
If $X$ and $Y$ in $\bM_2(\IR)$ are parallel,  then there are $U, V \in \bU_2(\IR)$
  with  $\det(U)  = 1$
such that  $X =U\diag(a_1, a_2)V^{\TT}$ and $Y =U\diag(b_1, b_2)V^{\TT}$ with $a_1\geq |a_2|$ and
 $|b_1|\geq |b_2|$. Let $R = V^{\TT} F_1 U$, and then
$$
 {\rm (1)} \ R = \begin{pmatrix} \cos\theta & -\sin \theta \cr
 \sin \theta & \cos\theta\end{pmatrix} \quad \hbox{ or } \quad
 {\rm (2)} \ R = \begin{pmatrix} \cos\theta & \sin \theta \cr
 \sin \theta & -\cos\theta\end{pmatrix}\quad \text{for some real $\theta$}.
$$
 Note that $U^{\TT}(E_{12}-E_{21})U = E_{12}-E_{21}$.
 Then
 $$\tr(F_1X) = \tr(F_1 U(a_{1}E_{11}+a_2E_{22})V^{\TT}) = \tr(R(a_{1}E_{11}+a_2E_{22}))$$
 and
 $$\tr(F_2X) = \tr(F_1(E_{12}-E_{21}) U(a_{1}E_{11}+a_2E_{22})V^{\TT})
 = \tr(R(E_{12}-E_{21})(a_{1}E_{11}+a_2E_{22})).$$
 Thus, $(\tr(F_1X), \tr(F_2X)) = (a_1+a_2) (\cos \theta, \sin \theta)$
 or $(a_1-a_2) (\cos\theta, -\sin \theta)$ depending on
 $R$ has the form (1) or (2).
 Similarly, $(\tr(F_1Y), \tr(F_2Y)) = (b_1+b_2) (\cos \theta, \sin \theta)$
 or $(b_1-b_2) (\cos\theta, -\sin \theta)$ depending on
 $R$ has the form (1) or (2).
 Therefore, $T(X)$ and $T(Y)$ are linearly dependent, and thus
 parallel. If $T$ has the form in (a.2), then $T$ preserves parallel pairs by Lemma \ref{lem:2x2-suff}.

For the necessity, suppose $T$ is a  linear map preserving parallel pairs.
If $T(P) = 0$ for every $P \in \bU_2(\IR)$, then $T$ is a zero map.
We assume below that $T(P) \ne 0$ for some $P \in \bU_2(\IR)$.
Replacing $T$ by a map of the form $UT(PX)V/\gamma$,
we may assume that $T(I) = \diag(1,a)$ with $1 \ge a \ge 0$.

If $T(\mathcal{R}_1) = T(\mathcal{R}_2)=0$, then
$T(a_0I + a_1\mathcal{R}_1 + a_2 \mathcal{R}_2 + a_3 \mathcal{R}_3) = a_0 T(I) + a_3 T(\mathcal{R}_3)$ has the form in (a.1)
with $F_1=I_2$, $F_2=(E_{12}-E_{21})$, $Z_1=\frac{1}{2}T(I)$ and $Z_2=\frac{-1}{2}T(\mathcal{R}_3)$.
Suppose it is not the case.
We may further assume that   $T(\mathcal{R}_2) \ne 0$, for else we can replace $T$ by the
map $X\mapsto  T(X\mathcal{R}_3)$.

{\bf Case 1.} Assume that $T(I) = \diag(1,a)$ with $1 > a \ge 0$.
For any $t \in \IR$, the real symmetric matrix
$I + t\mathcal{R}_2$ is parallel to $I$, and thus $T(I) + tT(\mathcal{R}_2)$ is parallel to $T(I)=\diag(1, a)$.
It follows from   Lemma \ref{2X2}(d) that $T(\mathcal{R}_2)= y_2\diag(1, a)\neq 0$ for some scalar $y_2$.
Similarly, $T(\mathcal{R}_1)=y_1\diag(1,a)$ for some scalar $y_1$.
Now, $\mathcal{R}_2 + t\mathcal{R}_3$ is parallel to $\mathcal{R}_2$ for any $t \in \IR$. Repeating the above argument, we see that
$T(\mathcal{R}_3)$ is a multiple of $T(\mathcal{R}_2)$, and thus
also a multiple of $\diag(1,a)$. Hence $T(A) = (\tr FA)Z$  for some nonzero $F\in \bM_2(\IR)$ and $Z=\diag(1,a)$.

{\bf Case 2.} Assume that $T(I) = I$. Furthermore, we may assume that
$T(P)$ is a multiple of an orthogonal matrix for any $P \in \bU_2(\IR)$; otherwise, we are back to Case 1. Since $X = qI + \mathcal{R}_3$ is
always a multiple of an orthogonal matrix in $\bU_2(\IR)$, so is  $T(X) = qI + T(\mathcal{R}_3)$
for any $q\in \IR$. Thus, $T(\mathcal{R}_3) = rI + s\mathcal{R}_3$ for some real numbers $r,s$ by taking $A=-iT(\mathcal{R}_3)$  in Lemma \ref{2X2}(b).

Suppose $X = a_1 \mathcal{R}_1 + a_2 \mathcal{R}_2$ such that $T(X) = Y \ne 0$.
Since $X$ is a multiple of an orthogonal matrix in $\bU_2(\IR)$, so is $Y$.
Moreover,   $X$ and $I$ are parallel, and so are
 $Y$ and $I$. Thus, by  Lemma \ref{2X2}(a), $Y$ is real symmetric.
 Since $Y$ is a multiple of an  orthogonal matrix in $\bU_2(\IR)$, we see that

\medskip\centerline{
(i) \ $Y= \gamma I$,  \qquad  or  \qquad (ii) \ $Y= \gamma Q\mathcal{R}_1 Q^{\TT}\in \bU_2(\IR)$,}

\noindent
for some $\gamma\in \IR$.

\medskip

Assume   $T(\mathcal{R}_2) = \beta I\neq0$.
Note that the real symmetric matrix $\mathcal{R}_1 + t\mathcal{R}_2$ is a multiple of an orthogonal matrix in $\bU_2$ for any real $t$. Thus, $T(\mathcal{R}_1) = \delta I$. Note also that $q\mathcal{R}_3 +  \mathcal{R}_2$ and $\mathcal{R}_2$ are always parallel, and so are $q T(\mathcal{R}_3) + T(\mathcal{R}_2)= qT(\mathcal{R}_3) + \beta I$ and $T(\mathcal{R}_2)= \beta I$.
Being a multiple of an orthogonal matrix, $T(\mathcal{R}_3) = rI + s \mathcal{R}_3$ has zero skew symmetric part. We see that $s = 0$ and $T(\mathcal{R}_3) = rI$. Hence, $T(\mathcal{R}_j)$ is a multiple of $I$ for $j = 1,2,3$. Thus $T(A) = (\tr FA)I$  for some nonzero $F \in \bM_2(\IR)$.

The case when $T(\mathcal{R}_1)=\alpha I\neq 0$ is similar.

\medskip
Assume   that $T(a_1 \mathcal{R}_1 + a_2 \mathcal{R}_2) = Z$ is always symmetric with trace zero.
Otherwise, we are back to the situation in (i).  So, the matrix representation of  $T$ with respect to the basis
$\bR=\{I/\sqrt 2, \mathcal{R}_1/\sqrt 2, \mathcal{R}_2/\sqrt 2, \mathcal{R}_3/\sqrt 2\}$ has the form
$$\begin{pmatrix}
 1 & 0 & 0& r\cr
0 & a_{11} & a_{12} & 0\cr
0 & a_{21} & a_{22}& 0 \cr
0 & 0 & 0 & s \cr
\end{pmatrix}.$$
\end{proof}

\begin{proof}[Proof of Theorem \ref{thm-parallel-R} {\rm (a)}]
For the sufficiency, if $X$ and $Y$ attain the triangle equality with $\|X\|=\|Y\|=1$, then there are $U, V \in \bU_2(\IR)$ with
 $\det(U)   = 1$ such that
 $X =U(E_{11} + a E_{22})V^{\TT}$ and $Y =U(E_{11} + b E_{22})V^{\TT}$ with
 $a, b \in [-1,1]$.

Suppose $T$ has the form (a.1). As in the proof of Theorem \ref{thm-parallel-R}(b), we see that
\begin{align*}
(\tr(F_1X), \tr(F_1Y)) &= \cos \theta (1+a, 1+b)
\ \text{or}\
\cos \theta (1-a, 1-b),\  \text{or} \\
(\tr(F_2X), \tr(F_2Y)) &= \sin \theta (1+a, 1+b)
\ \text{or}\
 -\sin \theta (1-a, 1-b)
\end{align*}
for some real scalar $\theta$.  Thus
\begin{align*}
T(X)&=(1\pm a)\cos\theta Z_1 \pm (1\pm a)\sin\theta Z_2,\\
T(Y)&=(1\pm b)\cos\theta Z_1 \pm (1\pm b)\sin\theta Z_2
\end{align*}
satisfy $(1\pm b)T(X)=(1\pm a)T(Y)$.  In particular, $T(X)$ and $T(Y)$ attain the triangle equality.  Then we see that
 $T$ preserves TEA pairs.

Suppose $T$ has the form in (a.2). We first consider the case when $T$ is invertible, that is,
$[T]_{\bR}=\diag(d_1, d_2, d_3, d_4)$ with $d_1d_2d_3d_4\neq 0$. We may replace $T$ by
the map $A \mapsto T(A)/d_1$, and assume that $d_1 = 1$.
Suppose that  $\|X+Y\| = \|X\|+\|Y\|$ for some    $X,Y\in\bM_2(\IR)$.
We can write $X= U(a_1E_{11} + a_2 E_{22})V^{\TT}$
and $Y =  U(b_1 E_{11}+  b_2 E_{22})V^{\TT}$
for some orthogonal matrices $U, V \in \bM_2(\IR)$, and some $ a_1,a_2,b_1,b_2\in \IR$ with $ a_1\geq a_2 \geq 0$ and $ b_1\geq |b_2|$.
 Let $G_1=T(UE_{11}V^{\TT})$ and $G_2= T(UE_{22}V^{\TT})$.
 As in the proof of Lemma \ref{lem:2x2-suff}, we see that $G_1,G_2$ satisfy the assumption in Lemma \ref{scheme}.
More precisely, we have orthogonal matrices $\tilde R,\tilde S$ in $\bM_2(\IR)$ such that
 $G_1 = \tilde R(s_1 E_{11} + s_2 E_{22})\tilde S$
and $G_2 = \zeta \tilde R(s_2 E_{11}+ s_1 E_{22})\tilde S$ for
$s_1 > s_2 \ge 0$, and a scalar  $\zeta = \pm 1$.
Observe that
\begin{eqnarray*}
T(X) & = & a_1G_1 + a_2 G_2\\
&=& \tilde R((a_1s_1+\zeta a_2s_2)E_{11} + (a_1s_2+\zeta a_2s_1)E_{22})\tilde S \\
&=& \tilde R(\mu_1 E_{11} + \mu_2 E_{22})\tilde S,
\end{eqnarray*}
where $\mu_1 =  a_1s_1+\zeta a_2s_2$ and
$\mu_2 =a_1s_2+\zeta a_2s_1$.
Similarly,
$$T(Y) =b_1 G_1 + b_2 G_2 = \tilde R(\nu_1 E_{11} + \nu_2 E_{22}) \tilde S,\ \hskip .5in \
$$
where
$$
\nu_1 = b_1s_1+\zeta b_2s_2\quad\text{and}\quad
\nu_2 = b_1s_2+\zeta b_2s_1.
$$

If $\zeta=1$, then $\mu_1 = a_1s_1+  a_2s_2 \geq a_1s_2+  a_2s_1 = \mu_2\geq 0$ and
$\nu_1 = b_1s_1+  b_2s_2 \geq |b_1s_2+ b_2s_1|=|\nu_2|$.
If $\zeta=-1$, then $\mu_1 = a_1s_1-  a_2s_2 \geq |a_1s_2-  a_2s_1| = |\mu_2|$ and
$\nu_1 = b_1s_1-  b_2s_2 \geq |b_1s_2-b_2s_1|=|\nu_2|$ again.
In both cases, we have
$$
\|T(X)+T(Y)\| = \mu_1 + \nu_1 = \|T(X)\| + \|T(Y)\|.
$$

Now, if $T$ with the form in (a.2) is singular so that some of the
$d_j = 0$ for $j \in \{1,2,3\}$.
We can approximate $T$ by a sequence $\{T_k\}_{k\geq1}$ of invertible TEA preservers carrying
the form \eqref{eq:(a.2)}.  It follows from Lemma \ref{lem:limit} that $T$ is a TEA preserver as well.

For the necessity, if $T$ is a nonzero linear map preserving  TEA pairs, it will preserve parallel pairs.
By Theorem \ref{thm-parallel-R}(b), we see that $T$ has the forms
given in (a,1), (a.2) or $T(A)=\tr(F A)Z$ for some nonzero $F, Z\in \bM_2(\IR)$. Suppose that $T(A)=\tr(F A)Z$.
There are $U, V \in \bU_2(\IR)$ such that $F =U(aE_{11} + b E_{22})V^{\TT}$ with $a\geq  b \geq 0$.
Let $X=V(-\mathcal{R}_1)U^{\TT}$ and $Y=VI_2U^{\TT}$.
Note that $X$ and $Y$ attain the triangle equality, and so are $T(X)=(-a+b)Z$ and $T(Y)=(a+b)Z$. This implies
$$
2b\|Z\|=\|T(X)+T(Y)\|=\|T(X)\|+\|T(Y)\|=2a\|Z\|,
$$
and $a=b>0$. Thus, $T$ has the form (a.1) with $F_1=\frac{1}{a}F$, $F_2=F_1(E_{12}-E_{21})$, $Z_1=aZ$ and $Z_2=0$.
\end{proof}

\section{Related results and questions}\label{S:future}

In connection to our results,
one can consider other norms on $\bM_{m,n}$. Suppose $1 \le k \le \min\{m,n\}$.
Define the Ky-Fan $k$-norm on $\bM_{m,n}$ by
$|A|_k = \sum_{j=1}^k s_j(A)$, the sum of the $k$ largest singular values.
If $k = 1$, it reduces to the spectral norm.
In \cite{Ket}, the authors consider the $(\bV, |\cdot|) = (\bM_{n}, |\cdot|_k)$,
for $2 \le k \le n$,
and show that an invertible linear map preserves
parallel pairs if and only if it preserves TEA pairs, and
the map must be a positive multiple of an isometry for $|\cdot|_k$.

\medskip
For the description of the isometries for the Ky-Fan $k$-norms on $\bM_{m,n}(\IF)$, one may see \cite{Li}.
Especially, it  includes two
special maps $\tau_1, \tau_2$ when $\bM_{m,n}(\IF) =\bM_4(\IR)$.
The proofs in \cite{Ket} for square matrices can be readily extended to $\bM_{m,n}$.
The authors of the paper  \cite{Ket} also characterize TEA and parallel preservers
with respect to the $k$-numerical radius on square matrices in a forthcoming paper.

\medskip
Our study can be considered for more general operator spaces or operator algebras.
For example,
one may consider extensions of our results  to
$B(H,K)$, the set  of bounded linear operators
$A: H \rightarrow K$ between real or complex Hilbert space over,
equipped with the operator norm
$$
\|A\| = \sup\{ \|Ax\|: x \in H, \ \|x\| = 1\}.
$$
In a forthcoming paper \cite{LTWW-poa}, we will give the characterization of linear parallel pair
 preservers of bounded operators,
as well as extensions to $C^*$-algebras and $JB^*$-triples.

\medskip
It would also be interesting to consider our problems for
other (finite or infinite dimensional) normed spaces; e.g., see \cite{Li}.
In fact, one may  consider the problems for linear maps
between two normed spaces $(\bV_1, |\cdot|_1)$ and
$(\bV_2, |\cdot|_2)$. An ambitious project is to determine those norms
$|\cdot|_1$ and $|\cdot|_2$ such that bijective linear maps
preserving parallel pairs or TEA pairs
must be  multiples of isometries for $|\cdot|_1$ and $|\cdot|_2$.

One may also consider problems related to the parallel pairs,
or TEA pairs.
For instance,  the Daugavet equation for operators $A$ in $B(H)$ is
\begin{equation}\label{D-eq}
\|A + I\| = \|A\| + 1,
\end{equation}
which is a well studied problem in function and operator theory
(see, e.g., \cite{A91,AA91, Werner97}).
One may consider linear maps of $B(H)$ sending the set of $A$ satisfying
(\ref{D-eq}) into or onto itself.

\bigskip
\noindent
\textbf{Acknowledgment}

Li is an affiliate member of the Institute for Quantum Computing, University of Waterloo; his research was partially supported by the Simons Foundation Grant 851334.
M.-C. Tsai, Y.-S. Wang and N.-C. Wong are supported by Taiwan NSTC grants 113-2115-M-027-003,
113-2115-M-005-008-MY2  and 112-2115-M-110-006-MY2,
respectively.

\end{document}